\documentclass[12pt, reqno]{amsart}

\usepackage[T1]{fontenc}
\usepackage[cp1250]{inputenc}

\usepackage{amsfonts}
\usepackage{amsmath}
\usepackage{amsthm}
\usepackage{amsaddr}
\usepackage{mathrsfs}
\usepackage[usenames]{color}
\usepackage{graphicx}
\usepackage[shortlabels]{enumitem}
\usepackage{fixltx2e}
\usepackage{xparse}

\usepackage{fullpage}
\usepackage{hyperref}

\newtheoremstyle{named}{}{}{\itshape}{}{\bfseries}{.}{.5em}{\thmname{#1} $\thmnote{ #3}_{\thmnumber{ #2}}$}

\newtheoremstyle{claim}{}{}{\itshape}{}{\bfseries}{.}{.5em}{}

\newtheorem{defi}{Definition}[section]
\newtheorem{przykl}[defi]{Example}
\newtheorem{twr}[defi]{Theorem}
\newtheorem{stwr}[defi]{Proposition}
\newtheorem{lt}[defi]{Lemma}

\newtheorem{uw}[defi]{Remark}

\theoremstyle{named}
\newtheorem{Ass}{Assumption}

\theoremstyle{claim}
\newtheorem*{claim}{Claim}
\newtheorem{claim_no}{Claim}

\newenvironment{assum}[2]{\begin{Ass}[#1] \label{#2} }{\end{Ass}}

\NewDocumentEnvironment{proofclaim}{m}{}{\begin{flushleft}%
This proves the Claim.
\end{flushleft} %
\bigskip
}

\newcommand{\ud}{\mathrm{d}}

\newcommand{\rzecz}{\mathbb{R}}
\newcommand{\wpo}{W^{1,p}(\Omega)}
\newcommand{\pd}[2]{\frac{\partial #1}{\partial #2}}
\newcommand{\dpd}[2]{\frac{\mathrm{d} #1}{\mathrm{d} #2}}
\newcommand{\mt}[1]{\mathrm{#1}}
\newcommand{\ttt}[1]{\textnormal{#1}}
\newcommand{\tto}[1]{\textrm{ #1 }}
\newcommand{\comz}[1]{\int\limits_{\Omega} #1 \ud z}
\newcommand{\essinf}{\mathop{\rm ess\,inf}}

\newcommand{\nabbibg}[6]{#1, \textit{#2}, #3, \textbf{#4} (#5), pp. #6} 
\newcommand{\nabbibk}[5]{#1, #2, #3, #4, #5} 

\newcommand*{\fullref}[1]{\hyperref[{#1}]{\nameref*{#1}\textsubscript{{\ref*{#1}}}}}

\numberwithin{equation}{section}

\begin{document}
\pagestyle{plain}
\title{
Nodal and multiple solutions for a~nonhomogeneous Neumann boundary problem
}
\author{Liliana Klimczak}
\address{Faculty of Mathematics and Computer Science, Jagiellonian University,
Lojasiewicza 6, 30-048 Krakow, Poland}
\begin{abstract}
We consider a nonlinear Neumann problem driven by a $p$-Laplacian-type, nonhomogeneous elliptic differential operator and a~Carath\'eodory reaction term. In this paper we prove the existence of two extremal constant sign smooth solutions and a nontrivial nodal smooth solution. In the proof we use variational methods with truncation techniques, critical point theory and Morse theory (critical groups). 
\end{abstract}
\maketitle
\section{Introduction}
Let $\Omega\subseteq\rzecz^N$ be a bounded domain with a $C^{2}$ boundary $\partial\Omega$. In this paper we look for smooth solutions to the following
Neumann problem
\begin{equation}\label{Problem}
\left\{
\begin{array}{ll}
 -\ttt{div }a\left(\nabla u (z)\right)=f\left(z, u (z)\right)& \tto{a.e. in}\Omega, \\
 \pd{u}{n_a}=0 &\tto{on} \partial\Omega,
\end{array}
\right.
\end{equation}
where $\pd{u}{n_a}=\left( a(\nabla u (z)),n(z)\right)_{\rzecz^N}$ with $n(\cdot)=(n_1(\cdot), \ldots, n_N(\cdot))$ the outward unit normal vector on $\partial
\Omega$. On the continuous map $a=(a_i)_{i=1}^N\colon\rzecz^N\to\rzecz^N$ we impose certain conditions (see Section \ref{sec:setting}) to obtain a $p$-Laplacian
type operator, which unifies several important differential operators. Similar conditions are studied widely in literature (see Damascelli \cite{damas98},
Montenegro \cite{monte99}, Gasi\'nski-Papageorgiou \cite{cGasinskiPapageorgiou2008a}, Motreanu-Papageorgiou \cite{MotreanuPapageorgiou2011} and also \cite{Klimczak2015b}), as they allow us to apply the regularity results of Lieberman \cite{liebe}. The reaction term
$f\colon\Omega\times\rzecz\to\rzecz$ is a Carath\'eodory function. We assume that $f(z, \cdot)$ has a positive and negative $z$-dependant zero and we impose some growth conditions on $f(z, \cdot)$ only
near zero, without any control in $\pm\infty$, and we use
the existence result for constant sign positive and negative solutions of problem \eqref{Problem} in \cite{Klimczak2015a} to prove the existence of extremal positive and negative solutions. Next we need to strengthen our hypotheses to obtain a third, nodal solution - some control on the behaviour of the reaction term in infinity is necessary.

The rest of this paper is organized as follows. In section \ref{sec:math_back} we provide mathematical preliminaries and  recall the main mathematical tools
which will be employed in this paper. In section \ref{sec:setting} we formulate the assumptions on maps $a$ and $f$, provide some examples and formulate the main theorem of the paper. Next, in section \ref{sec:extremal}, we recall the existence result and main tools of the proof and then prove the existence of two extremal solutions. Using this result, in section \ref{sec:nodal} we provide the proof of the existence of the nodal solution.
\section{Mathematical background}\label{sec:math_back}
In this section we provide the main mathematical tools needed in the proofs. We will denote by $\left(\cdot,\cdot \right)_{\rzecz^N}$ the scalar product in $\rzecz^N$ and by $|\cdot |_{N}$ - the Lebesgue measure in $\rzecz^N$. 
\begin{twr}[25.D in Zeidler \cite{ziib}]\label{twr:minicoer}
Let $X$ be a reflexive Banach space and let $M\subseteq X$ be its nonempty closed convex subset. Suppose that $\phi\colon M\to\rzecz$ is a weakly sequentially lower semicontinuous and weakly coercive functional, i.e. for each $u\in M$ and each sequence $\{u_n\}_n\subseteq M$
 such that $u_n\rightarrow u$ weakly in $X$, we have
$$
  \phi(u)\leq \liminf_{n\to\infty}\phi(u_n);
$$ and $$
\lim\limits_{\|u\|\to\infty}\phi(u)=\infty \quad \tto{on} M.
$$ Then $\phi$ has a minimum on $M$.
\end{twr}

\begin{twr}[1.7 in Lieberman \cite{liebe}]\label{twr:liebe}
Let $h\colon\rzecz_+\to\rzecz$  be a $C^1$-function satisfying
$$
\delta<\frac{th'(t)}{h(t)}\leq c_0\tto{ for all }t>0
$$
with some constants $\delta>0$, $c_0>0$.
We define $H(\xi)=\int\limits_0^{\xi}h(t)\ud t$. By $W^{1, H}(\Omega)$ we denote the class of functions which are weakly differentiable in the set $\Omega$ with
$$
\comz{H(|\nabla u|)}<\infty.
$$
Let $\alpha \in (0,1]$, $\Lambda$, $\Lambda_1$, $M_0>0$ be positive constants and let $\Omega\subseteq \rzecz^N$ be a bounded domain with $C^{1,\alpha}$ boundary. Suppose that $A=(A_1, \ldots,
A_N)\colon\Omega\times[-M_0, M_0]\times\rzecz^N\to\rzecz^N$ is differentiable, $B\colon\Omega\times\left([-M_0, M_0]\times\rzecz^N\right)\to\rzecz$ is a Carath\'eodory
function and functions $A$, $B$ satisfy the following conditions
\begin{subequations}\label{ass:lieberman}
\begin{gather}
 \left(\nabla_y A(z_1,\xi_1, y)x, x\right)_{\rzecz^N}\geq \frac{h(|y|)}{|y|}|x|^2, \quad  y\neq 0_N \label{eq:liebea}\\
 |\frac{\partial}{\partial y_j} A_i(z, \xi, y)|\leq \Lambda\frac{h(|y|)}{|y|},  \quad  y\neq 0_N \label{eq:liebeb}\\
|A(z_1,\xi_1, y )-A(z_2,\xi_2, y )| \leq \Lambda_1\left(1+h(|y|) \right)\left(|z_1-z_2|^{\alpha}+|\xi_1-\xi_2|^{\alpha}\right),\label{eq:liebec}\\
|B(z_1, \xi_1, y)|\leq \Lambda_1(1+h(|y|)|y|),\label{eq:liebed}
\end{gather}
\end{subequations}
for all $z_1, z_2\in\Omega$, $\xi_1, \xi_2\in[-M_0, M_0]$ and $x,y\in\rzecz^N$.
Then any $W^{1, H}(\Omega)$ solution $u$ of
\begin{equation}
\ttt{div } A(z, u, \nabla u)+B(z, u, \nabla u)=0
\end{equation}
 in $\Omega$ with $|u|\leq M_0$ in $\Omega$ is in $C^{1,\beta}(\Omega)$ for some positive $\beta$ depending on $\alpha$, $\Lambda$, $\delta$, $c_0$, $N$.
\end{twr}

\begin{twr}[5.3.1 in Pucci-Serrin \cite{pserr2007}]\label{twr:pserr}
Let $\Omega\subseteq\rzecz^N$ be a domain. Suppose that $A\in C^1(\rzecz^+)$ is such that function $s\mapsto sA(s)$ is strictly increasing in $\rzecz^+$ and $s
A(s)\to 0$ as $s\to 0^+$. Let $B\in L^{\infty}_{loc}(\Omega\times\rzecz^+\times \rzecz^N)$ satisfy the following condition
\begin{equation}\label{eq:B2}
B(z, \xi, y)\geq - \kappa \Phi (|y|) - p(\xi) \tto{for all}(z,\xi,y)\in \Omega\times[0,\infty)\times \rzecz^N, \tto{such that }|\xi|\leq 1,
\end{equation}
where $\kappa>0$ is a constant, $p\colon\rzecz^+\to\rzecz$ is non-decreasing on some interval $(0,\delta)$, $\delta>0$, $\Phi(s):=sA(s)$ when $s>0$ and $\Phi(0):=0$. For $s\geq 0$ we define
\begin{equation}
L(s)=s\Phi(s)-\int_0^s\Phi(t)\ud t.
\end{equation}
If either $p\equiv 0$ in $[0, d]$, $d>0$, or the following condition is satisfied
\begin{equation}
\lim_{\epsilon\to 0^+}\int_0^{\epsilon}\frac{1}{L^{-1}(P(s))}\ud s=\infty,
\end{equation}
where $P(s)=\int_0^s p(t)\ud t$, then the strong maximum principle for
\begin{equation}\label{eq:pserr}
\ttt{div }(A(|\nabla u(z)|)\nabla u(z))+ B(z, u(z), \nabla u(z))\leq 0
\end{equation}
holds,
i.e. if u is a classical distribution solution of \eqref{eq:pserr} with $u(z_0)=0$ at some point $z_0\in\Omega$, then $u\equiv 0$ in
$\Omega$. By classical distribution solution we mean a function $u\in C^1(\Omega)$, which satisfies \ \eqref{eq:pserr} in the distribution sense.
 \end{twr}
In what follows $\|\cdot \|$ denotes the norm in Sobolev space
$W^{1,p}(\Omega)$. We will assume that $1<p<\infty$.

In the analysis of problem \eqref{Problem} we will use the positive cone
$$C_+=\{u\in C^1(\bar{\Omega}) \mid u(z)\geq 0 \tto{for all} z\in\bar{\Omega} \tto{and} \pd{u}{n_a}=0 \tto{on} \partial\Omega\}$$
and its interior given by
$$
\ttt{int } C_+=\{u\in C_+ \mid u(x)>0\tto{for all} z\in\bar{\Omega}\}.
$$
To deal with the boundary condition in problem \eqref{Problem}, we introduce the following function space framework, due to Casas-Fern\'andez \cite{CF1989}:
for $p' \in(1, \infty)$ such that $\frac{1}{p}+\frac{1}{p'}=1$ we introduce a separable Banach space
$$
W^{p'}(\mt{div}, \Omega)=\{v\in L^{p'}(\Omega, \rzecz^N)\mid \ttt{div } v\in L^1(\Omega)\},
$$
endowed with the norm
$$
\|v\|_{W^{p'}(\mt{div}, \Omega)}=\|v\|_{L^{p'}(\Omega, \rzecz^N)}+\|\ttt{div } v\|_{L^1(\Omega)}.
$$
If $\Omega$ has a Lipschitz boundary $\partial\Omega$, we have that the space $C^{\infty}(\bar{\Omega}, \rzecz^N)$ is dense in $W^{p'}(\mt{div}, \Omega)$ (see Lemma 1. in Casas-Fern\'andez \cite{CF1989}). We denote the space of
traces on $\partial \Omega$ by $W^{1/p', p}(\partial\Omega)$, endowed with the usual norm, and denote the trace of $u\in\wpo$ on $\partial\Omega$ by
$\gamma_0(u)$. Let us also consider the space
$$
T^p(\partial\Omega)=W^{1/p',p}(\partial\Omega)\cap L^{\infty}(\Omega)
$$
endowed with the norm $\|h\|_{T^p(\partial\Omega)}=\|h\|_{W^{1/p',p}(\partial\Omega)}+\|h\|_{ L^{\infty}(\Omega)}$. We denote the dual space of
$T^p(\partial\Omega)$ by $T^{-p'}(\partial\Omega)$ and the duality brackets by $\langle \cdot,\cdot \rangle_T$. We have
$$T^p(\partial\Omega)=\{\gamma_0(u)\mid u\in \wpo\cap L^{\infty}(\Omega)\}.$$
Also there exists a unique linear continuous map
$$
\gamma_n\colon W^{p'}(\mt{div}, \Omega)\to T^{-p'}(\partial\Omega),
$$
such that
$$
\gamma_n(v)=\left(v, n\right)_{\rzecz^N}, \ \forall v \in C^{\infty}(\bar{\Omega}, \rzecz^N).
$$
From this result one can obtain the following Green's formula.
\begin{twr}[1 in Casas-Fern\'andez \cite{CF1989}]\label{twr:green}
Let $a=(a_i)_{i=1}^N\colon\Omega\times(\rzecz\times\rzecz^N)\to\rzecz^N$ be a Carath\'eodory map, which satisfies
$$
  |a_i(z, s, \xi)|\leq k_1(|s|^{p-1}+|\xi|^{p-1})+k_2(z), \ i=1,\ldots, N
$$
with some constant $k_1>0$ and  a function $k_2\in L^{p'}(\Omega)$. Then if $u\in\wpo$ and $-\ttt{div } a(\cdot, u, \nabla u)\in L^1(\Omega)$, then there exists
a~unique element of $T^{-p'}(\partial\Omega)$, which by extension we denote $\partial u/\partial n_a$, satisfying the Green's formula:
$$
  \sum_{i=1}^N\comz{a_i(z, u(z), \nabla u(z))\frac{\partial v}{\partial z_i}}
  =\comz{-\ttt{div }a(z, u(z), \nabla u(z))v(z)}+\left< \frac{\partial u}{\partial n_a}, \gamma_0(v)\right>_T
$$
for all $v\in \wpo\cap L^{\infty}(\Omega) $.
\end{twr}
To obtain the third, nodal solution of problem \eqref{Problem}, we will make use of the notion of a critical point of mountain pass type and a result known in the literature as \emph{mountain pass theorem}  (see for example  Gasi\'nski-Papageorgiou \cite[p. 649]{GP_NA}). To formulate this result, we introduce the following  definition of compactness-type condition.
\begin{defi}
Let $X$ be a Banach space and $\phi\in C^1(X)$. We say that $\phi$ satisfies the Cerami condition at level $c \in R$, if any sequence $\{x_n\}_{n\geq 1}\subseteq X$, such that
$$
\phi(x_n) \ \to \ c 
\quad \tto{and} \quad
(1+\|x_n\|_X)\phi'(x)_n \ \rightarrow \ 0 \ \tto{in} \ X^*,
$$
has a strongly convergent subsequence. If this is true at every level $c\in\rzecz$,
then we simply say that $\phi$ satisfies the Cerami condition.
\end{defi}
 For $\phi\in C^1(X)$ and $c\in\mathbb{R}$  we define the following sets:
\begin{gather*}
  \phi^c    \  := \ \big\{ x\in X:\ \phi(x)\leq c\big\},\\
  K_{\phi}  \  := \ \big\{ x\in X:\ \phi'(x)=0\big\},\\
  K_{\phi}^c  \  :=  \  \big\{ x\in K_{\phi}:\ \phi(x)=c\big\}.
\end{gather*}
\begin{twr}\label{twr:mpt}
Let $X$ be a Banach space. If $\phi\in C^1(X)$, $x_0, x_1\in X$ with $\|x_0-x_1\|_X>r>0$,
$$
\max\{\phi(x_0), \phi(x_1)\}\leq \inf\limits_{\|x-x_0\|_X=r}\phi(x)
$$
and $\phi$ satisfies the Cerami condition at level $c$, where
$$
c:=\inf\limits_{\gamma\in\Gamma}\max\limits_{t\in[0,1]}\phi(\gamma(t))
$$
and
$$
\Gamma:=\{\gamma\in C([0,1],X)\mid \gamma(0)=x_0, \gamma(1)=x_1\},
$$
then $c\geq \inf\limits_{\|x-x_0\|_X=r}\phi(x)$ and $c$ is a critical value of $\phi$ ($K^c_{\phi}\neq\emptyset$).
\end{twr}
\begin{defi}[6.98 in Motreanu-Motreanu-Papageorgiou \cite{MotreanuMotreanuPapageorgiou2014}]
Let $X$ be a Banach space, $\phi\in C^1(X)$ and let $x\in X$ be a critical point of $\phi$ ($x\in K_{\phi}$). We say that $x$ is of \emph{mountain pass type}, if, for any open neighbourhood $U$ of $x$, the set $\{y\in U \mid \phi(y)<\phi(x)\}$ is nonempty and not path-connected.
\end{defi}
Critical groups of $\phi\in C^1(X)$ at an isolated critical point $u\in K^c_{\phi}$ are defined by
$$
C_k(\phi,u)=H_k(U\cap \phi^c,U\cap \phi^c\backslash \{u\}), \quad \forall k\geq 0,
$$
where $c:=\phi(u)$ and $U\subseteq X$ is an open neighbourhood of $u$ such that $K_{\phi}\cap \phi^c \cap U=\{u\}$.  The excision property of the singular homology implies that the preceding
defnition of critical groups is independent of the particular choice of the
neighborhood $U$ (see Definition 6.43 and Remark 6.44 in Motreanu-Motreanu-Papageorgiou \cite{MotreanuMotreanuPapageorgiou2014}).
\begin{stwr}[6.100 in Motreanu-Motreanu-Papageorgiou \cite{MotreanuMotreanuPapageorgiou2014}]\label{stwr:cg1_mpt}
Let $X$ be a reflexive Banach space,
$\phi\in C^1(X)$, and $u\in K_{\phi}$
isolated critical point with $c:=\phi(u)$ isolated critical value in
$\phi(K_{\phi})$. If $u$ is of mountain pass type, then
$C_1(\phi,u)\neq 0$.
\end{stwr}
For a Carath\'eodory function $f\colon\Omega\times\rzecz\to\rzecz$, let
$N_f\colon\wpo\to\mathcal{M}(\Omega,\rzecz)$, where $\mathcal{M}(\Omega,\rzecz)$ is the set of all measurable functions on $\Omega$, be defined by
$$
N_f(u)\colon \Omega \ni x\mapsto N_f(u)(x)=f(x,u(x))\in \rzecz \ \tto{for} u\in\wpo.
$$
\section{Problem setting}\label{sec:setting}
In this section we formulate our assumptions and provide some examples. We end the section with the main result of this paper.
 We  will consider the following hypotheses on the map $a$:
\begin{assum}{H(a)}{ass:ha1:zero} There exists a function $a_0\colon[0,\infty)\to[0,\infty)$ with  $a_0\in C^1(0,\infty)$,  $a_0\in C([0,\infty))$  and $a_0(t)>0$ for $t>0$ such that
\begin{gather}
a(y)=a_0(|y|)y, \quad \tto{for all}y\in\rzecz^N\label{eq:ass:ha1:zero}.
\end{gather}
\end{assum}
\begin{assum}{H(a)}{ass:ha1:nabla} There exist some constants $\delta, c_0, c_1, c_2, c_3>0$,  $q\in(1,p)$ and a~function $h\in C^1(0,\infty)$  satisfying
\begin{gather}
\delta<\frac{th'(t)}{h(t)}\leq c_0\tto{ for all }t>0,\label{eq:ass:ha1:h1}\\
c_1 t^{p-1} \leq h(t) \leq c_2(t^{q-1}+t^{p-1}) \tto{for all }t>0\label{eq:ass:ha1:h2},
\end{gather}
such that
\begin{equation}
\label{eq:ass:ha1:nabla} |\nabla a(y)|\leq c_3\frac{h(|y|)}{|y|}\tto{for all}y\in\rzecz^N\backslash \{0\}.
\end{equation}
\end{assum}
\begin{assum}{H(a)}{ass:ha1:nabla_sc} For all $y, \xi \in\rzecz^N$ such that $y\neq 0$ we have
\begin{equation}
 \left(\nabla a(y)\xi,\xi \right)_{\rzecz^N}\geq \frac{h(|y|)}{|y|}|\xi|^2, \label{eq:ass:ha1:nabla_sc}
\end{equation}
where $h\in C^1(0,\infty)$ is as in \emph{$\fullref{ass:ha1:nabla}$}. \end{assum}
\begin{assum}{H(a)}{ass:ha1:mu} There exists some constants   $\tau \in(1, p]$ and $\mu \in(1,q]$ such that 
\begin{equation}
\label{eq:ass:ha1:tau1}
\tto{the map} t\to G_0(t^{1/\tau}) \tto{is convex on} (0,\infty)  
\end{equation}
and
\begin{equation}
\label{eq:ass:ha1:mu1}
\lim\limits_{t\to 0^{+}}\frac{G_0(t)}{t^{\mu}}=0,
\end{equation}
where
$$G_0(t)=\int\limits_0^ta_0(s)s\ud s.$$
\end{assum}
\begin{assum}{H(a)}{ass:ha1:apG}
For any $y\in\rzecz^N$, we have that
\begin{equation}\label{eq:ass:ha1:eq:apG}
\left(a(y), y\right)_{\rzecz^N}\leq p\ G(y).
\end{equation}
\end{assum}
\begin{stwr}\label{stwr:propG-a}
Let
$$
G(y):=G_0(|y|), \quad y\in\rzecz^N.
$$ If assumptions \emph{$\fullref{ass:ha1:zero}$}--\emph{$\fullref{ass:ha1:mu}$} hold, 
then  $G$ is strictly convex, $G(0)=0$ and $\nabla G(y)=a(y)$ for $y\in\rzecz^N\backslash \{0\}$, thus $a$ is strictly monotone. Moreover,  there exists $c_4>0$ such that
    \begin{equation}\label{eq:w1_a1b}
    |a(y)|\leq c_4(|y|^{q-1}+|y|^{p-1})\quad \tto{for all}y\in\rzecz^N;
    \end{equation}
   also
    \begin{equation}\label{eq:w1_a1c}
    \left(a(y), y \right)_{\rzecz^N}\geq \frac{c_1}{p-1}|y|^p\quad \tto{for all}y\in\rzecz^N.
    \end{equation}
Inequalities \eqref{eq:w1_a1b} and \eqref{eq:w1_a1c}, together with the Cauchy-Schwarz inequality, imply that
\begin{equation}\label{eq:osz}
\frac{c_1}{p(p-1)}|y|^p\leq G(y)\leq c_4(|y|^q+|y|^p)\quad \tto{for all}y\in\rzecz^N.
\end{equation}
Also, the nonlinear map
  $A\colon \wpo\longrightarrow \wpo^*$ defined by
\begin{equation}\label{gwiazdka}
\left<A(u), v\right>=\int_{\Omega}\left(a(\nabla u(x)),\nabla v(x))\right)_{\rzecz^N}\ud x, \quad u, v\in \wpo,
\end{equation}  
  is bounded, continuous
  and of type $(S)_+$,
  i.e., if
$$
  u_n
  \ \longrightarrow\ u
  \quad\textrm{weakly in} \ \wpo
$$
  and
$$
  \limsup_{n\rightarrow +\infty} \big\langle A(u_n),\ u_n-u\big\rangle
  \ \leq\ 0,
$$
  then $u_n\longrightarrow u$ in $\wpo$.
\end{stwr}
The proof of Proposition \ref{stwr:propG-a} can be found in \cite{Klimczak2015a} (see also  Gasi\'nski-Papageorgiou \cite{cGasinskiPapageorgiou2008a}, Proposition 3.1).

\begin{przykl}\label{exampl:a}
  Here we present some examples of maps satisfying hypotheses $H(a)$:
\begin{enumerate}[(a)]
\item  $a(y)=|y|^{p-2}y$ with $1<p<\infty$. This map corresponds to the $p$-Laplacian operator defined by
$$
\Delta_pu \ = \ \mt{div}(|\nabla u|^{p-2}\nabla u), \quad u\in \wpo.
$$
\item $a(y)=|y|^{p-2}y+|y|^{q-2}y$ with $1<q<p<\infty$. This map corresponds to the $(p,q)$-differential operator defined by
$$
\Delta_pu+\Delta_qu, \quad u\in \wpo.
$$
\item $a(y)=(1+|y|^2)^{(p-2)/2}y$ with $1<p<\infty$. This map corresponds to the generalized $p$-mean curvature differential operator defined by
$$
\mt{div}((1+|\nabla u |^2)^{(p-2)/2}\nabla u), \quad u\in \wpo.
$$
\end{enumerate}
  \end{przykl}
  \begin{uw}
The main aim of the hypotheses $H(a)$ is to unify several operators operators, which are widely examined due to their applications in physics (see for example \cite{aDrabek2007a} for the $p$-Laplacian or \cite{aBenciAvenia2000a} for the $(p,q)$-differential operator). This kind of the hypotheses comes from the regularity theorem of Lieberman (Theorem \ref{twr:liebe}) - in the case $A(z, \xi, y)=a(y)$,  the assumptions \eqref{eq:liebea} -- \eqref{eq:liebec} simplify to the following two conditions:
\begin{subequations}
\begin{gather}
 \left(\nabla a(y)x, x)\right)_{\rzecz^N}
 \ \geq \
 \frac{h(|y|)}{|y|}|x|^2, \quad  y\neq 0_N; \label{eq:liebea_s}\\
 |\frac{\partial}{\partial y_j} a_i(y)|
 \ \leq\ 
 \Lambda\frac{h(|y|)}{|y|},  \quad  y\neq 0_N. \label{eq:liebeb_s}
\end{gather}
\end{subequations}
Thus, assuming on the map $a$ hypotheses \emph{$\fullref{ass:ha1:nabla}$} -- \emph{$\fullref{ass:ha1:nabla_sc}$}, for a~suitable reaction term $f$ we can easily obtain that all bounded weak solutions of problem \eqref{Problem} actually have locally H\"older continuous first derivative.
\end{uw}	
Let $f_0\colon\Omega\times\rzecz\to\rzecz$ be a Carath\'eodory function with subcritical growth in the second variable, i.e. there exisits $\alpha\in L^{\infty}(\Omega)_+$ and $\beta>0$ such that
$$
|f_0(z, \xi)|
\ \leq \ 
\alpha(z)+ \beta |\xi|^{r-1}, \ \tto{for a.e.} z\in \Omega, \tto{all} \xi\in\rzecz
$$ 
with $r\in[1,p^*)$, where
$$
p^*
  \ =\
\left\{
  \begin{array}{lll}
  \displaystyle \frac{Np}{N-p} & \textrm{if} & p<N,\\
    +\infty & \textrm{if} & p\geq N
  \end{array}
\right.
$$  
is the Sobolev critical exponent. Let $F_0(z, \xi)=\int_0^{\xi}f_0(z, s)\ud s$ and define the $C^1$-functional $\phi_0\colon\wpo\to\rzecz$ by
$$ 
\phi_0(u)=\comz{G(\nabla u(z))}-\comz{F_0(z, u(z))}, \ u\in\wpo.
$$
With minor modifications, the proof of Theorem 3.1 in Motreanu-Papageorgiou \cite{MotreanuPapageorgiou2011} can be adapted to obtain the following result.
\begin{twr}\label{twr:minimizers} If hypotheses \emph{$\fullref{ass:ha1:zero}$}--\emph{$\fullref{ass:ha1:mu}$} hold and 
if $u_0\in\wpo$ is a local $C^1(\bar{\Omega})$-minimizer of $\phi_0$, i.e. there exists $r_0>0$ such that
$$
\phi_0(u_0)\leq \phi_0(u_0+h) \ \tto{for all} \ h\in C^1(\bar{\Omega}), \|h\|_{C^1(\bar{\Omega})}\leq r_0,
$$
then $u_0\in C^1(\bar{\Omega})$ and it is a $\wpo$-minimizer of $\phi_0$, i.e. there exists $r_1>0$ such that
$$
\phi_0(u_0)\leq \phi_0(u_0+h) \ \tto{for all} \ h\in \wpo, \|h\|\leq r_1.
$$
\end{twr}
Our assumptions on the Carath\'eodory map $f\colon \Omega\times\rzecz\to\rzecz$ are the following
\setcounter{Ass}{0}
\begin{assum}{H(f)}{ass:hf1:ainf} For a.e. $z\in\Omega$, $f(z,0)=0$ and for every $\rho>0$ there exists $a_{\rho}\in L^{\infty}(\Omega)_+$ such that for a.e.~$z\in\Omega$ and every $\xi\in\rzecz$ we have
\begin{equation}  |\xi|\leq \rho \ \Rightarrow \ |f(z,\xi)| \leq a_{\rho}(z); \label{eq:ass:hf1:ainf}
\end{equation}
\end{assum}
\begin{assum}{H(f)}{ass:hf1:wc}
There exist functions $w_{\pm}\in\wpo\cap C(\bar{\Omega})$ and constants $c_-$, $c_+$ such that
\begin{gather}
w_-(z)\leq c_-<0<c_+\leq w_+(z) \tto{for all} z\in\bar{\Omega};\label{eq:ass:hf1:wc-c+}\\
f(z,w_+(z) )\leq 0\leq f(z,w_-(z) ) \tto{for a.e.} z\in\Omega;\\
A(w_-)\leq 0 \leq A(w_+) \tto{in}W^{1,p}(\Omega)^* \label{eq:Aw-0Aw+},
\end{gather}
where $A$ is defined by \eqref{gwiazdka} and by \eqref{eq:Aw-0Aw+} we mean that for any $u\in\wpo$ with $u\geq 0$, the following inequalities hold:
$$
\left<A(w_-), u\right> \leq 0 \ \tto{and} \ \left<A(w_+), u\right> \geq 0.
$$
\end{assum}
\begin{assum}{H(f)}{ass:hf1:eq:flF}
There exists $\delta_0>0$, such that for a.e. $z\in\Omega$  and for all $\xi\in\rzecz$ such that $0<|\xi|\leq \delta_0$, we have
\begin{equation}\label{eq:ass:hf1:eq:flF}
0< f(z, \xi)\xi\leq \mu F(z, \xi) \quad \tto{and} \quad \essinf\limits_{\Omega} F(\cdot, \delta_0)>0
\end{equation}
with $F(z,\xi)=\int\limits_0^\xi f(z,t)\ud t$ and $\mu$ is as in \emph{$\fullref{ass:ha1:mu}$};
\end{assum}
\begin{assum}{H(f)}{ass:hf1:eq:fgeqx}
There exist $\widehat{c}_0, \widehat{c}_1>0$, and $s, r\in\rzecz$ with $s\neq r$ and $s<\mu$, $s\leq\tau\leq p\leq r < p^*$ (where $\tau$ and $\mu$ are the
same as in \emph{$\fullref{ass:ha1:mu}$}) such that
\begin{equation}\label{eq:ass:hf1:eq:fgeqx}
f(z,\xi)\xi\geq \widehat{c}_0|\xi|^s-\widehat{c}_1|\xi|^r \quad \tto{for all}\xi\in\rzecz \tto{and for a.e.} z\in\Omega.
\end{equation}
\end{assum}
\begin{assum}{H(f)}{ass:hf1:bdd}
There exist $\widehat{c}_3>0$ such that
\begin{equation}\label{eq:ass:hf1:eq:bdd}
|f(z,\xi)|\leq  \widehat{c}_3(1+|\xi|^{q-1})
\end{equation}
for a.e. $z\in\Omega$ and for all $\xi\in\rzecz$, where $q\in(1,p)$ is the same as in $\fullref{ass:ha1:nabla}$.
\end{assum}

\begin{uw}\label{rem:Fconc}
Hypothesis \emph{$\fullref{ass:hf1:eq:flF}$} implies that
\begin{equation}\label{eq:Fconc}
\widehat{c}_2|\xi|^{\mu}\leq F(z, \xi) \quad \tto{for a.e.} z\in\Omega \tto{and for } |\xi|\leq \delta_0,
\end{equation}
with some $\widehat{c}_2>0$.
\end{uw}

\begin{uw}\label{rem:psi}
Let us consider
\begin{equation}
\psi(\xi)=\widehat{c}_0 |\xi|^{s-2} \xi-\widehat{c}_1 |\xi|^{r-2}\xi, \quad \xi\in\rzecz.\label{eq:defpsisr}
\end{equation}
Then inequality \eqref{eq:ass:hf1:eq:fgeqx} becomes
$ f(z,\xi)\xi\geq\psi(\xi)\xi$ for $\xi\in\rzecz$ and a.e. $z\in\Omega$, so for  a.e. $z\in\Omega$ we have
\begin{equation}\label{eq:osz_psi}
\begin{array}{ll}
f(z,\xi)\geq \psi(\xi),& \tto{ if }\xi \geq 0,\\
f(z,\xi)\leq \psi(\xi),& \tto{ if }\xi < 0.\\
\end{array}
\end{equation}
Let $\xi_0=\left(\frac{\widehat{c}_0}{\widehat{c}_1}\right)^{\frac{1}{r-s}}$. We can observe that  $\psi>0$ on $(-\infty, -\xi_0)\cup(0, \xi_0)$ and $\psi<0$
on $(-\xi_0, 0)\cup(\xi_0, \infty)$. Also $\psi$ is strictly increasing on $(\infty,-\xi_0)$ and strictly decreasing on $(\xi_0, \infty)$. Thus, from \emph{$\fullref{ass:hf1:wc}$} and
\eqref{eq:osz_psi} we infer that for a.e. $z\in\Omega$ we have $\psi(w_+(z))<0$ and $\psi(w_-(z))>0$.
\end{uw}
\begin{lt}\label{lt:srp}  Let $1<s\leq p\leq r$ with $s<r$. For any constants $\alpha, \beta, \gamma$, given, if $\alpha, \beta>0$, then we can find $M_1, M_2>0$ such that for any $\xi>0$ we have
\begin{equation}\label{eq:srp}
\alpha \xi^s-\beta\xi^r+\gamma \xi\leq M_1 - M_2 \xi^p.
\end{equation}
\end{lt}
\begin{przykl}\label{exampl:f} Function
$f(z, \xi)=\psi(\xi)=\widehat{c}_0 |\xi|^{s-2} \xi-\widehat{c}_1 |\xi|^{r-2}\xi$ satisfies conditions
\emph{$\fullref{ass:hf1:ainf}$}--\emph{$\fullref{ass:hf1:eq:fgeqx}$} (see \cite{Klimczak2015a}). 
\end{przykl}
\begin{uw} We use the hypotheses   \emph{$\fullref{ass:ha1:apG}$} and  \emph{$\fullref{ass:hf1:bdd}$} only in the proof of Proposition \ref{stwr:nodal}, which is needed only in the proof of the existence of the nontrivial nodal solution. Thus the result concerning existence of the two constant sign solutions and two extremal solutions remains valid with only assuming hypotheses \emph{$\fullref{ass:ha1:zero}$}--\emph{$\fullref{ass:ha1:mu}$} and \emph{$\fullref{ass:hf1:ainf}$}--\emph{$\fullref{ass:hf1:eq:fgeqx}$}. In particular, in these results on the reaction therm we impose growth conditions  with respect to the second variable  only near $0$, without any control in $\pm\infty$. Examples \ref{exampl:a} and \ref{exampl:f} satisfy all the hypotheses $H(a)$ and $H(f)$.
\end{uw}
Now we can state the main result of this paper.
\begin{twr}\label{twr:main_theorem}
If hypotheses \emph{$\fullref{ass:ha1:zero}$}--\emph{$\fullref{ass:ha1:mu}$} and \emph{$\fullref{ass:hf1:ainf}$}--\emph{$\fullref{ass:hf1:eq:fgeqx}$} hold,
then problem \eqref{Problem} has a smallest positive smooth solution
$$
u_+\in \textrm{int } C_+ 
$$
and a biggest negative smooth solution
$$
v_-\in -\textrm{int }C_+.
$$
If we additionally assume that hypotheses \emph{$\fullref{ass:ha1:apG}$} and \emph{$\fullref{ass:hf1:bdd}$} hold, then the problem \eqref{Problem} has a third nontrivial and nodal solution $y_0\in[v_-, u_+]\cap C^1(\Omega)$.
\end{twr}
\section{Existence of two extremal constant sign solutions}\label{sec:extremal}
In this section we prove that problem \eqref{Problem} has two extremal solutions of constant sign. First we recall an existence theorem, which can be found in \cite{Klimczak2015a}. For the convenience of the reader, we will remind main steps of the proof. For a detailed proof we refer to \cite{Klimczak2015a} (see also the proof of Proposition \ref{stwr:extremality}).
\begin{stwr}\label{stwr:2cs_sol}
If hypotheses \emph{$\fullref{ass:ha1:zero}$}--\emph{$\fullref{ass:ha1:mu}$} and \emph{$\fullref{ass:hf1:ainf}$}--\emph{$\fullref{ass:hf1:eq:fgeqx}$} hold,
then problem \eqref{Problem} has at least two nontrivial constant sign smooth solutions
$$
u_0\in \textrm{int } C_+\tto{and} v_0\in -\textrm{int }C_+.
$$
\end{stwr}
\begin{proof}
We will describe the proof of the existence of a nontrivial positive smooth solution $u_0$. The proof for the negative smooth solution is similar. We introduce the following truncation of the reaction term:
\begin{equation}\label{eq:ftr1}
\widehat{f}_+(z,\xi)=\left\{\begin{array}{ll}
0 & \tto{if} \xi<0,\\
f(z,\xi)+\xi^{p-1} & \tto{if} 0\leq \xi \leq w_+(z),\\
f(z, w_+(z))+ \xi^{p-1}+\psi(\xi) -\psi( w_+(z)) & \tto{if} \xi >w_+(z)
\end{array}\right.
\end{equation}
where $\psi$ is given by \eqref{eq:defpsisr}. This is a Carath\'eodory function. Let$$
\widehat{F}_+(z,\xi)=\int_0^{\xi}\widehat{f}_+(z,t)\textrm{d}t
$$
and consider the $C^1-$functional $\widehat{\varphi}_+\colon\wpo\to\rzecz$, defined by
\begin{equation}
\widehat{\varphi}_+(u)=\int_{\Omega}G(\nabla u(z))\ud z+\frac{1}{p}\|u\|^p_p-\int_{\Omega}\widehat{F}_+(z,u(z))\ud z, \quad u\in
W^{1,p}(\Omega).
\end{equation}

Functional $\widehat{\varphi}_+$ is weakly sequentially lower semicontinuous by the strict convexity of $G$. Using $\fullref{ass:hf1:ainf}$, $\fullref{ass:hf1:wc}$ and $\fullref{ass:hf1:eq:fgeqx}$ together with Remark \ref{rem:psi}, we prove that there exist some constants $\widehat{c}_4, \widehat{c}_6>0 $ and $\widehat{c}_5, \widehat{c}_7\in \rzecz$  such that
\begin{equation*}\begin{split}
\widehat{\varphi}_+(u)&=\int_{\Omega}G(\nabla u(z))\ud z+\frac{1}{p} \|u\|_p^{p}-\comz{\widehat{F}_+(z,u(z))}\\
 &\geq \frac{c_1}{p(p-1)}\comz{\|\nabla u(z)\|^p}+ \frac{1}{p} \|u^+\|_p^{p}+\frac{1}{p} \|u^-\|_p^{p}+\widehat{c}_4\|u^+\|_p^p-\frac{1}{p} \|u^+\|_p^{p}+\widehat{c}_5 \\
 & \geq \widehat{c}_6\|u\|^p+\widehat{c}_7.
\end{split}\end{equation*}
Hence  $\widehat{\varphi}_+$ is also weakly coercive, so by Theorem \ref{twr:minicoer}  there
 exists $u_0\in W^{1,p}(\Omega)$ such that
$$
\widehat{\varphi}_+(u_0)=\min\limits_{u\in W^{1,p}(\Omega)}\widehat{\varphi}_+(u).
$$
Thus we have that $(\widehat{\varphi}_+)'(u_0)=0$ (see Zeidler \cite{ziib}, Proposition 25.11, p. 510). This implies
\begin{equation}\label{eq:arnf+}
A(u_0)+ |u_0|^{p-2}u_0=N_{\widehat{f}_+}(u_0).
\end{equation}
Using  $\fullref{ass:ha1:mu}$ and  $\fullref{ass:hf1:eq:flF}$ with Remark \ref{rem:Fconc} we show that $\widehat{\varphi}_+(u_0)<0$, so $u_0\neq 0$. Acting on \eqref{eq:arnf+} first with $-u_0^-\in \wpo$ and next with $(u_0-w_+)^+\in \wpo$, by $\fullref{ass:hf1:wc}$,  Remark \ref{rem:psi} and the strict monotonicity of $a$ (see Proposition \ref{stwr:propG-a}), we obtain that
$
u_0\in[0, w_+],
$
where $[0, w_+]=\{u\in \wpo \mid 0 \leq u(z) \leq w_+(z) \tto{for a.e.} z\in\Omega \}$.
Then, by \eqref{eq:ftr1} and \eqref{eq:arnf+} we obtain
\begin{equation}\label{eq:arf1}
A(u_0)=N_{f}(u_0).
\end{equation}
 Also $u_0\in L^{\infty}(\Omega)$, as for a.e. $z\in \Omega$ we have $|u(z)|\leq w_+(z)\leq \|w_+\|_{C(\bar{\Omega})}$.
Next we prove that
\begin{equation}\label{eq:solu0_ae}
-\ttt{div }a\left(\nabla u_0 (z)\right)=f\left(z, u_0 (z)\right) \tto{a.e. in}\Omega
\end{equation}
and, using the Green's formula for operator $a$ (see Theorem \ref{twr:green}), that
$$
\frac{\partial u}{\partial n_a}=0 \tto{in} T^{-p'}(\partial\Omega)
$$
(see Gasi\'nski-Papageorgiou \cite{cGasinskiPapageorgiou2008a}).
Then, by the regularity result of Lieberman (Theorem \ref{twr:liebe}) and the strong maximum principle (see Theorem
\ref{twr:pserr}), we obtain that $u_0\in\ttt{int }C_+$ solves \eqref{Problem}.

\medskip
To prove the existence of the nontrivial negative smooth solution, we use the  truncation:
\begin{equation}\label{eq:ftr2}
\widehat{f}_-(z,\xi)=\left\{\begin{array}{ll}
0 & \tto{if} \xi>0,\\
f(z,\xi)+\xi^{p-1} & \tto{if}  w_-(z)\leq \xi\leq 0 ,\\
f(z, w_-(z))+ \xi^{p-1}+\psi(\xi) -\psi( w_-(z)) & \tto{if} \xi <w_-(z)
\end{array}\right.
\end{equation}
and proceed analogously.
\end{proof}

To show that problem \eqref{Problem} has extremal solutions of constant sign, we consider the following auxiliary problem
\begin{equation}\label{AuxProblem}
\left\{
\begin{array}{l}
 -\ttt{div }a\left(\nabla u (z)\right)=\psi(u(z)) \tto{ in}\Omega, \\
 \pd{u}{n_a}=0 \tto{on} \partial\Omega,
\end{array}
\right.
\end{equation}
(see \eqref{eq:defpsisr}). Recall that $1<s\leq \tau\leq p\leq r <p^*$ with $s\neq r$ and $s<\mu$.

We also introduce the following integral functional $\sigma_+\colon L^1(\Omega)\to \rzecz$, defined by
$$
\sigma_+(u)=\left\{
\begin{array}{l}
\comz{G(\nabla u^{\frac{1}{\tau}})} \tto{ if } u\geq 0,\ u^{\frac{1}{\tau}}\in\wpo  \\
 \infty \quad  \tto{otherwise}.
\end{array}
\right.
$$
\begin{lt}\label{lt:s0cone}
$\sigma_+$ is an element of the cone of proper (i.e. $\ttt{dom } \sigma_0:=\{u\in L^1(\Omega)\mid \sigma_+(u)<\infty\}\neq\emptyset $) convex and lower semicontinuous functions, denoted by $\Gamma_0( L^1(\Omega))$.
\end{lt}
\begin{proof}
Using $\fullref{ass:ha1:mu}$, Proposition \ref{stwr:propG-a} and Lemma 1 from Diaz-Saa \cite{DiazSaa1987}, we proceed like in the proof of Proposition 7 in Gasi\'nski-Papageorgiou \cite{GasinskiPapageorgiou2014}.
\end{proof}
\begin{stwr}\label{stwr:aux_u_s}
If hypotheses \emph{$\fullref{ass:ha1:zero}$} -- \emph{$\fullref{ass:ha1:mu}$} hold, then problem \eqref{AuxProblem} has a unique nontrivial positive solution 
$u_*\in \ttt{int }C_+$ and due to the oddness of\eqref{AuxProblem}, $v_*=-u_*\in-\ttt{int }C_+$ is the unique nontrivial negative solution.
\end{stwr}
\begin{proof}
Observe that $u_*\equiv \xi_0\in \ttt{int } C_+$ solves \eqref{AuxProblem} (see Remark \ref{rem:psi}).
Let us check the uniqueness of this positive solution.

Let $u, y \in \wpo$ be two nontrivial positive solutions of \eqref{AuxProblem}. Then, by Theorem \ref{twr:green}, we have
\begin{gather}
A(u)=N_{\psi}(u) \quad  \tto{on}  \wpo\cap L^{\infty}(\Omega);\label{eq:AuNpsi}\\
A(y)=N_{\psi}(y) \quad  \tto{on}  \wpo\cap L^{\infty}(\Omega).\label{eq:AyNpsi}
\end{gather}
 We will prove that $u, y$ are bounded above by $w_+$. To this end, first we show that $u, y\geq \xi_0$. Suppose, to derive a contradiction, that $|\{u<\xi_0\}|_N:=|\{z\in\Omega\mid u(z)<\xi_0\}|_N>0$. Then we act on \eqref{eq:AuNpsi} with $(\xi_0-u)^+\in\wpo$ and obtain
\begin{displaymath}\begin{split}
0&<\int\limits_{\{u<\xi_0\}}\psi(u(z))(\xi_0-u)(z)\ud z=\comz{\psi(u(z))(\xi_0-u)^+(z)}=\comz{\left(a(\nabla u, \nabla (\xi_0-u)^+)\right)_{\rzecz^N}}\\
&=-\int\limits_{\{u<\xi_0\}}\left(a(\nabla u, \nabla u)\right)_{\rzecz^N}\ud z \leq -\int\limits_{\{u<\xi_0\}}\frac{c_1}{p(p-1)}|\nabla u(z)|^p \ud z\leq 0
\end{split}\end{displaymath}
(see Remark \ref{rem:psi} and \eqref{eq:w1_a1c}), which is a contradiction. So we have that $u\geq \xi_0$. In the same manner we show that $y\geq \xi_0$. Next we suppose that $|\{u>w_+\}|_N:=|\{z\in\Omega\mid u(z)>w_+(z)\}|_N>0$ and we act on \eqref{AuxProblem}  with $(u-w_+)^+\in \wpo$. We obtain
\begin{displaymath}\begin{split}
\int_{\Omega}\left(a(\nabla u(z)),\nabla (u-w_+)^+(z))\right)_{\rzecz^N}\ud z &=\int_{\Omega}\psi(u(z))(u-w_+)^+(z)\ud z <0\\
&\leq\int_{\Omega}\left(a(\nabla w_+(z)),\nabla (u-w_+)^+(z))\right)_{\rzecz^N}\ud z
\end{split}\end{displaymath}
(see Remark \ref{rem:psi} and $\fullref{ass:hf1:wc}$). Hence 
$$
\int_{\Omega}\left(a(\nabla u(z))-a(\nabla w_+(z)),\nabla (u-w_+)^+(z))\right)_{\rzecz^N}\ud z<0,
$$
a contradiction with the monotonicity of $a$ (see Proposition \ref{stwr:propG-a}). So we have $u\in[\xi_0, w_+]$ and in the same way we obtain $y\in [\xi_0, w_+]$. Thus $u, y\in L^{\infty}(\Omega)$. Then the nonlinear regularity theory (\cite{liebe} p. 320) implies that $u, y \in \ttt{int } C_+$. Also $u, y \in \ttt{dom }\sigma_+$. Let $x\in C^1(\bar{\Omega})$. Then there exists $\lambda_0\in(0,1)$ such that for any $|\lambda|<\lambda_0$ we have $\left(u^\tau+\lambda  x\right)\in\ttt{int }C_+$ and also $\left(u^\tau+\lambda  x\right)\in\ttt{dom }\sigma_+$. Hence the G\^{a}teaux derivative of $\sigma_+$ at $u^{\tau}$ in the direction $x$ exists and, via the chain rule, we have
$$
\sigma_+'(u^{\tau})(x)=\comz{\frac{-\ttt{div }a(\nabla u)}{u^{\tau-1}}x} \ \tto{and} \ \sigma_+'(y^{\tau})(x)=\comz{\frac{-\ttt{div }a(\nabla y)}{y^{\tau-1}}x}.
$$
The convexity of $\sigma_+$ implies the monotonicity of $\sigma_+'$. Hence
\begin{displaymath}
\begin{split}
0\leq & \left<\sigma_+'(u^{\tau})-\sigma_+'(y^{\tau}),u^{\tau}-y^{\tau} \right>=\\
&=\comz{\left(\frac{-\ttt{div }a(\nabla u)}{u^{\tau-1}}+\frac{\ttt{div }a(\nabla y)}{y^{\tau-1}}\right)(u^{\tau}-y^{\tau})}\\
&=\comz{\left(\frac{\hat{c}_0 u ^{s-1}-\hat{c}_1 u^{r-1}}{u^{\tau-1}}-\frac{\hat{c}_0 y ^{s-1}-\hat{c}_1 y ^{r-1}}{y^{\tau-1}}\right)(u^{\tau}-y^{\tau})}\\
&=\comz{\left[\left(\hat{c}_0\frac{1}{u^{\tau -s}}-\hat{c}_1 u^{r-\tau}\right)-\left(\hat{c}_0\frac{1}{y^{\tau -s}}-\hat{c}_1 y^{r-\tau} \right)\right](u^{\tau}-y^{\tau})}\leq 0
\end{split}
\end{displaymath}
(see \ref{AuxProblem}). The last inequality holds because, as $s\leq \tau \leq r$ with $s<r$, the map $\zeta\to\hat{c}_0\frac{1}{\zeta^{\tau-s}}-\hat{c}_1\zeta^{r-\tau}$ is strictly decreasing on $(0, \infty)$ and $\zeta\to\zeta^{\tau}$ is strictly increasing. Therefore $u=y$ and this proves the uniqueness of the solution $u_*\in\ttt{int }C_+$.

The oddness of \eqref{AuxProblem} implies that $v_*:=-u_*\equiv\xi_0\in-\ttt{int }C_+$ is the unique nontrivial negative solution of \eqref{AuxProblem}. 
\end{proof}

Using this proposition, one can establish the existence of extremal nontrivial constant sign solutions for problem \eqref{Problem}.
\begin{stwr}\label{stwr:extremality}
If hypotheses \emph{$\fullref{ass:ha1:zero}$} -- \emph{$\fullref{ass:ha1:mu}$} and \emph{$\fullref{ass:hf1:ainf}$} -- \emph{$\fullref{ass:hf1:eq:fgeqx}$} hold, then problem \eqref{Problem} has the smallest nontrivial positive solution $u_+\in\ttt{int }C_+$ and the biggest nontrivial negative solution $v_-\in-\ttt{int }C_+$.
\end{stwr} 
\begin{proof}
Let $\mathcal{Y}_+$ be the set of nontrivial postive solutions of problem \eqref{Problem} in the ordered interval $[0, w_+]$. From Proposition \ref{stwr:2cs_sol}, we know that $\mathcal{Y}_+\neq \emptyset$.

We are going to use the Kuratowski-Zorn lemma to prove the existence of a~minimal element of $\mathcal{Y}_+$. We will need the following auxiliary results:
\begin{claim_no}\label{lt:minorantaa}
If $y\in\mathcal{Y}_+$, then $u_*\leq y$ with $u_*\in\ttt{int }C_+$ as in Proposition \ref{stwr:aux_u_s}.
\end{claim_no}
\begin{proofclaim}{lt:minoranta}
 Let $y\in\mathcal{Y}_+$. Then
 \begin{equation}\label{eq:arfy}
 A(y)=N_f(y) \quad \tto{on}  \wpo\cap L^{\infty}(\Omega).
 \end{equation}
  Recall that $u_*\equiv \xi_0$ (see Remark \ref{rem:psi}). Define $\Omega_1:=\{z\in\Omega  \mid y(z)<\xi_0$\}.
Suppose, to derive a contradiction, that $|\Omega_1|_{N}>0$. We act on \eqref{eq:arfy} with $(\xi_0-y)^+\in  \wpo\cap L^{\infty}(\Omega)$ and obtain
\begin{equation*}\begin{split}
0&<\int_{\Omega_1}\psi(y(z))(\xi_0-y)\ud z\leq \int_{\Omega_1}f(z, y(z))(\xi_0-y)\ud z=\int_{\Omega_1}(a(\nabla y(z)),\nabla(\xi_0-y(z)))_{\rzecz^N}\ud z\\
&=-\int_{\Omega_1}(a(\nabla y(z)),\nabla y(z))_{\rzecz^N}\ud z\leq - \int_{\Omega_1}\frac{c_1}{p-1}|\nabla y(z)|^p\ud z\leq 0
\end{split}\end{equation*}
(see Remark \ref{rem:psi} and \eqref{eq:w1_a1c} in Proposition \ref{stwr:propG-a}), which is a contradiction.
\end{proofclaim}
\begin{claim_no}\label{lt:minorantab} 
 $\mathcal{Y}_+$ is downward directed, i.e. for any $u_1, u_2 \in \mathcal{Y}_+$ one can find $u\in\mathcal{Y}_+$, such that
$$
u\leq u_1 \ \tto{and} \ u\leq u_2.
$$
\end{claim_no}
\begin{proofclaim}{lt:minorantab}
The idea of this proof follows from Gasi\'nski-Papageorgiou \cite[p. 208]{GasinskiPapageorgiou2010} and uses the notion of upper and lower solutions. We say that $x\in\wpo$ is a \emph{weak upper solution} (respectively, \emph{weak lower solution}) of problem \eqref{Problem}, if
\begin{gather*}
\langle A(x),h\rangle \ \geq \ \langle N_f(x),h\rangle \quad  \tto{for all} h\in\wpo, \  h\geq 0 \\
 (\langle A(x),h\rangle) \ \leq \ \langle N_f(x),h\rangle \quad   \tto{for all} h\in\wpo, \ h\geq 0 \ \ttt{ respectively}).
\end{gather*}
Using the monotonicity of the map $a$ (see Proposition \ref{stwr:propG-a}), we can adapt the proof of Lemma 4.2. in Gasi\'nski-Papageorgiou \cite{GasinskiPapageorgiou2009} to see that the set of upper solutions of \eqref{Problem} is downward directed and in fact, if $x_1, x_2\in\wpo$ are two upper solutions, then $\bar{x}:=\min\{u_1, u_2\}$ is an upper solution. Also the set of lower solutions of \eqref{Problem} is upward directed and if $x_1, x_2\in\wpo$ are two lower solutions, then $\underbar{x}:=\max\{u_1, u_2\}$ is a lower solution.

Now let $u_1, u_2 \in \mathcal{Y}_+$. We have that $\bar{u}:=\min\{u_1, u_2\}\in\wpo\cap L^{\infty}(\Omega)$ is an upper solution of \eqref{Problem}, as $u_1, u_2$ are clearly upper solutions, too.
We use the truncation
\begin{equation}\label{eq:tr:check_f}
\check{f}(z,\xi)=\left\{\begin{array}{ll}
0 & \tto{if} \xi<0,\\
f(z,\xi)+\xi^{p-1}  & \tto{if} 0\leq \xi \leq \bar{u}(z),\\
f(z, \bar{u}(z))+ \xi^{p-1}+\psi(\xi) -\psi( \bar{u}(z)) & \tto{if} \xi >\bar{u}(z)
\end{array}\right.
\end{equation}
and  we define a $C^1$-functional  $\check{\varphi}\colon\wpo\to\rzecz$ by 
$$
\check{\varphi}_+(u)=\int_{\Omega}G(\nabla u(z))\ud z+\frac{1}{p}\|u\|^p_p-\int_{\Omega}\check{F}(z,u(z))\ud z, \quad u\in
W^{1,p}(\Omega),
$$
with $\check{F}(z, \xi)=\int_0^{\xi}\check{f}(z, s)\ud s$.  Then we proceed as in the proof of Proposition \ref{stwr:2cs_sol}:
it is clear that $\check{\varphi}_+$ is weakly sequentially lower semicontinuous and we want to prove that it is also weakly coercive. We have
\begin{equation}\label{eq:tr:check_F}
\check{F}_+(z,\xi)=\left\{\begin{array}{lll}
0 & &\tto{if} \xi<0,\\
\int\limits_0^{\xi}f(z,t)\ud t+\frac{1}{p} \xi^{p} && \tto{if} 0\leq \xi \leq \bar{u}(z),\\
\int\limits_0^{\bar{u}(z)}f(z,t)\ud t \ +&\hspace{-0.3cm} (\xi-\bar{u}(z))f(z, \bar{u}(z)) \ + & \\
\hspace{0.5cm}+\int\limits_{\bar{u}(z)}^{\xi}\psi(t) \ud t& - \  (\xi-\bar{u}(z))\psi( \bar{u}(z))+\frac{1}{p} \xi^{p} & \tto{if} \xi >\bar{u}(z),
\end{array}\right.
\end{equation}
so
\begin{equation}\label{eq:coerF}\begin{split}
\comz{\check{F}_+(z, u(z))}
&=\comz{\check{F}_+(z, u^+(z))}
=\int\limits_{\{0<u< \bar{u} \}} \int\limits_0^{u(z)}f(z,t)\ud t \ud z
+ \int\limits_{\{u> \bar{u} \}} \int\limits_0^{\bar{u}(z)}f(z,t)\ud t \ud z\\
&+\int\limits_{\{u> \bar{u} \}}\big((u(z)-\bar{u}(z))f(z, \bar{u}(z))-   (u(z)-\bar{u}(z))\psi( \bar{u}(z)) \big)\ud z\\
&+\int\limits_{\{u> \bar{u} \}}\big(\frac{\widehat{c}_0}{s}u(z)^s-\frac{\widehat{c}_1}{r}u(z)^r-\frac{\widehat{c}_0}{s}\bar{u}(z)^s+\frac{\widehat{c}_1}{r}\bar{u}(z)^r\big)\ud z +\frac{1}{p}\|u^+\|_p^p.
\end{split}\end{equation}
By $\fullref{ass:hf1:ainf}$, as $\bar{u}\in  L^{\infty}(\Omega)$, we have that
\begin{equation}\label{eq:coerF1}\begin{split}
\int\limits_{\{0<u< \bar{u} \}}\!\!
\int\limits_0^{u(z)}f(z,t)\ud t \ud z
+ \!\!\int\limits_{\{u> \bar{u} \}}
\!\!\int\limits_0^{\bar{u}(z)}f(z,t)\ud t \ud z
&\leq  \!\!\!\!\!\! \int\limits_{\{0<u< \bar{u} \}}
\!\!\int\limits_0^{\|\bar{u}\|_{\infty}}f^+(z,t)\ud t \ud z
+  \!\!\int\limits_{\{u> \bar{u} \}}
\!\!\int\limits_0^{\|\bar{u}\|_{\infty}}f^+(z,t)\ud t \ud z\\
&\leq \int\limits_{\{0<u \}}\!\!
\!\!\int\limits_0^{\|\bar{u}\|_{\infty}} a_{\|\bar{u}\|_{\infty}}(z)\ud t \ud z \leq \check{c}_1, 
\end{split}\end{equation}
where $\check{c}_1\geq 0$ is a constant. As $\bar{u}\geq 0$ and for $\xi\geq 0$ we have $f(z, \xi)\geq \psi(\xi)$ for a.e. $z\in\Omega$ (see Remark \ref{rem:psi}), we obtain
\begin{equation}\label{eq:coerF2}
\int\limits_{\{u> \bar{u} \}}\bar{u}(z)(\psi( \bar{u}(z))-f(z, \bar{u}(z))) \ud z \leq 0.
\end{equation}
Since $\bar{u}\leq w_+$ and $w_+$ is a continuous function defined on a bounded subset of $\rzecz^N$, there exists some constant $\check{c}_2 \in \rzecz$, such that
\begin{equation}\label{eq:coerF3}
\int\limits_{\{u> \bar{u} \}}\frac{\widehat{c}_1}{r}\bar{u}(z)^r -\frac{\widehat{c}_0}{s}\bar{u}(z)^s \ud z \leq \check{c}_2.
\end{equation}
Using $\fullref{ass:hf1:ainf}$, Remark \ref{rem:psi} and Lemma \ref{lt:srp}, we can find constants $\check{c}_3, \check{c}_4, \check{c}_5>0$ and $\check{c}_6\in\rzecz$, such that
\begin{equation}\label{eq:coerF4}\begin{split}
&\int\limits_{\{u> \bar{u} \}}\big(\frac{\widehat{c}_0}{s}u(z)^s-\frac{\widehat{c}_1}{r}u(z)^r+u(z)(f(z, \bar{u}(z))-   \psi( \bar{u}(z)))\big)\ud z\\
\leq &\int\limits_{\{u> \bar{u} \}}\big(\frac{\widehat{c}_0}{s}u(z)^s-\frac{\widehat{c}_1}{r}u(z)^r+u(z)(a_{\|\bar{u}\|_{\infty}}-   \psi( \min\limits_{z\in\Omega} w_+(z)))\big)\ud z\\
\leq &\int\limits_{\{u> \bar{u} \}}(\check{c}_3- \check{c}_4 u(z)^p) \ud z\leq \check{c}_6 -\check{c}_5\|u^+\|^p_p.
\end{split}\end{equation}
In the last inequality we have used the fact that $w_+$ is a continuous function (see $\fullref{ass:hf1:wc}$) and thus
$$
\int\limits_{\Omega\backslash{\{u> \bar{u} \}}}u^+(z)^p \ud z=\int\limits_{\{0<u\leq  \bar{u} \}}u^+(z)^p \ud z\leq \int\limits_{\{0<u\leq  \bar{u} \}} w_+(z)^p\leq \check{M},
$$
where $\check{M}\geq 0$ is a constant. Combining   \eqref{eq:coerF1}, \eqref{eq:coerF2}, \eqref{eq:coerF3}, \eqref{eq:coerF4} with \eqref{eq:coerF} and
using \eqref{eq:osz}, we obtain
\begin{equation*}\begin{split}
\check{\varphi}_+(u)&=\int_{\Omega}G(\nabla u(z))\ud z+\frac{1}{p} \|u\|_p^{p}-\comz{\check{F}_+(z,u(z))}\\
 &\geq \frac{c_1}{p(p-1)}\comz{\|\nabla u(z)\|^p}
 + \frac{1}{p} \|u^+\|_p^{p}
 +\frac{1}{p} \|u^-\|_p^{p}+\check{c}_5\|u^+\|_p^p-\frac{1}{p} \|u^+\|_p^{p}+\check{c}_7 \\
 &\geq \check{c}_8 +\check{c}_9\|u\|^p,
\end{split}\end{equation*}
 with some constants
$\check{c}_7, \check{c}_8\in\rzecz$, $\check{c}_9>0$. Hence $\check{\varphi}_+$ is also weakly coercive, so we can apply Theorem \ref{twr:minicoer} and
obtain that there exists $\check{u}\in W^{1,p}(\Omega)$ such that
$$\check{\varphi}_+(\check{u})=\min\limits_{u\in\wpo}\check{\varphi}_+(u).$$
Hence
 $(\check{\varphi}_+)'(u_0)=0$ and thus
\begin{equation}\label{eq:arnf+-check}
A(\check{u})+ |\check{u}|^{p-2}\check{u}=N_{\check{f}_+}(\check{u}).
\end{equation}
To prove the nontriviality of $\check{u}$, we choose an arbitrary $\check{v}_+\in\ttt{int }C_+$ and we will find $\check{t}\in(0,1)$ such that
\begin{equation}\label{eq:phi_check_u}
\check{\varphi}(\check{u})\leq \check{\varphi}(t\check{v}_+)<0=\check{\varphi}(0).
\end{equation}
Using \eqref{eq:ass:ha1:mu1} (hypothesis  $\fullref{ass:ha1:mu}$), for a~given $\varepsilon >0$ we can find
$\delta_{1,\varepsilon}\in(0,\delta_0]$ such that
$$
G_0(t)\leq \varepsilon t^{\mu}\quad\forall t\in(0, \delta_{{1,\varepsilon}}].
$$
So, choosing $\check{t}\in(0,1)$ such that $\check{t}\check{v}_+(z)\leq \delta_{{1,\varepsilon}}$ and $|\nabla
(\check{t}\check{v}_+)(z)|\leq\min\{\delta_{0}, c_+\}$ for a.e. $z\in\Omega$, we have
$$
\check{\varphi}_+(\check{t}\check{v}_+)\leq \check{t}^{\mu}\varepsilon \|\nabla (\check{v}_+)\|^{\mu}_{\mu}- \widehat{c}_2 \check{t}^{\mu}\|\check{v}_+\|^{\mu}_{\mu}
$$
(see \eqref{eq:osz}, Remark \ref{rem:Fconc}, \eqref{eq:tr:check_f} and recall that by \eqref{eq:ass:hf1:wc-c+}, Claim \ref{lt:minorantaa} and Remark \ref{rem:psi} we have that $c_+\leq \xi_0 \leq \bar{u}$).
Choosing $\varepsilon< \frac{\widehat{c}_2\|\check{v}_+\|^{\mu}_{\mu}}{\|\nabla (\check{v}_+)\|^{\mu}_{\mu}}$, we obtain that
 \eqref{eq:phi_check_u} holds and thus $\check{u}\neq 0$. Acting with $-\check{u}^-\in\wpo$ on \eqref{eq:arnf+-check}, we get
\begin{equation*}\begin{split}
\frac{c_1}{p-1}\|\nabla(-\check{u}^-))\|_p^p+\|\!(-\check{u}^-)\|_p^p 
&\leq \comz{\left(a(\nabla \check{u}(z)),\nabla(-\check{u}^-(z)))\right)_{\rzecz^N}}
+\comz{|\check{u}(z)|^{p-2}\check{u}(z)(-\check{u}^-)(z)}\\
&=\int_{\Omega}\check{f}_+(z, \check{u}(z))(-\check{u})^-(z)\ud z =0
\end{split}\end{equation*}
(see \eqref{eq:w1_a1c} and \eqref{eq:tr:check_f}), so $\check{u}\geq 0$.
Suppose, to derive a contradiction, that $|\{\check{u}>\bar{u}\}|_{\rzecz^N}>0$. We have
\begin{equation*}\begin{split}
\langle A(\check{u}), (\check{u}-\bar{u})^+\rangle 
&+\comz{(\check{u}(z))^{p-1}(\check{u}-\bar{u})^+(z)}
=\comz{\check{f}(z,\check{u})(\check{u}-\bar{u})^+(z)}\\
&=\comz{\left(f(z,\bar{u}(z))+(\check{u}(z))^{p-1}+\psi(\check{u}(z))-\psi(\bar{u}(z))\right)(\check{u}-\bar{u})^+(z)}\\
&>\comz{(\check{u}(z))^{p-1}(\check{u}-\bar{u})^+(z)}+ \langle A(\bar{u}), (\check{u}-\bar{u})^+\rangle
\end{split}\end{equation*}
(here we have used the properties of $\psi$, see Remark \ref{rem:psi}, and the fact that $\bar{u}$ is an upper solution of \eqref{Problem}). Hence
$$
\langle A(\check{u})-A(\bar{u}),  (\check{u}-\bar{u})^+\rangle < 0,
$$
a contradiction with the strict monotonicity of $a$ (see Proposition \ref{stwr:propG-a}). So $\check{u}\in[0, \bar{u}]$ and thus $\check{u}\in\mathcal{Y}_+$ with $ \check{u} \leq u_1$ and $\check{u}\leq u_2$.
\end{proofclaim}
Let $C\subseteq \mathcal{Y}_+$ be a chain (i.e. totally ordered subset of $\mathcal{Y}_+$). From Dunford-Schwartz \cite{DunfordSchwartz1988} (p. 336) we know that we can find a sequence $\{u_n\}_{n\geq 1}\subseteq C$ such that
$$
\ttt{inf } C=\inf\limits_{n\geq 1}u_n.
$$
We have 
\begin{equation}\label{eq:aunnf}
A(u_n)=N_f(u_n)  , \quad n\geq 1.
\end{equation}
It follows from  Claim \ref{lt:minorantaa} that $u_*\leq u_n\leq w_+$, so the sequence $\{u_n\}_{n\geq 1}\subseteq \wpo$ is bounded, and consequently, it possesses a weakly convergent subsequence (see Leoni \cite{Leoni}, p. 302). Also
\begin{displaymath}
\begin{split}
&u_n \to u \tto{in} L^p(\Omega)\\
&\nabla u_n \rightharpoonup \nabla u \tto{in} L^p(\Omega;\rzecz^N)
\end{split}
\end{displaymath}
with some $u\in\wpo$ (see Leoni \cite{Leoni} pp. 302, 322). 

Now we act on \eqref{eq:aunnf} with $(u_n-u)\in \wpo$:
$$
\left<A(u_n),u_n-u\right>=\comz{f(z, u_n(z))(u_n-u) }.
$$

We have 
\begin{displaymath}\begin{split}
|\comz{f(z, u_n(z))(u_n-u) }|
& \leq
 \left(\comz{|f(z, u_n(z))|^{p'}}\right)^{1/p'}\|u_n-u \|_p\\
 &\leq \left(|\Omega|_N \|a_{w_+(z_0)}\|_{L^{\infty}}^p \right)^{1/p'}\|u_n-u \|_p,
\end{split}\end{displaymath}
so passing to the limit  as $n\to \infty$ we obtain
$$
\lim\limits_{n\to\infty}\left<A(u_n),u_n-u\right>=0,
$$
Thus, by Proposition \ref{stwr:propG-a},
$$
u_n\to u \tto{in} \wpo.
$$
Hence, passing to the limit as $n\to\infty$ in \eqref{eq:aunnf}, we have
$$
A(u)=N_f(u)\ \tto{and} \ u_*\leq u\leq w_+,
$$
so by Claim \ref{lt:minorantaa} we obtain that $u\in \mathcal{Y}_+$ is the lower bound for $C$. Since $C$ was an arbitrary chain in $\mathcal{Y}_+$, from the Kuratowski-Zorn lemma, we infer that $\mathcal{Y}_+$ has a minimal element $u_+\in\mathcal{Y}_+$.
From Claim \ref{lt:minorantab} we infer that $u_+$ is the smallest nontrivial positive solution of \eqref{Problem}.  The nonlinear regularity theory (\cite{liebe} p. 320) implies that $u_+\in\ttt{int }C_+$.

To prove the existence of the biggest nontrivial negative solution $v_-\in-\ttt{int }C_+$ we proceed analogously.
\end{proof}
\section{Nodal solution}\label{sec:nodal}
In this section we prove the existence of a third, nontrivial and nodal solution. In what follows, let $u_+\in\ttt{int }C_+$ be the smallest positive solution and  $v_-\in-\ttt{int }C_+$ - the biggest negative solution. We introduce the following truncation
\begin{equation}\label{eq:ftr3}
\widetilde{f}(z,\xi)=\left\{\begin{array}{ll}
f(x, v_-(z)) + |\xi|^{p-2}\xi +\psi(\xi)-\psi(v_-(z))& \tto{if} \xi<v_-(z),\\
f(z,\xi)+ |\xi|^{p-2}\xi & \tto{if} v_-(z)\leq \xi \leq u_+(z),\\
f(z, u_+(z))+ |\xi|^{p-2}\xi+\psi(\xi) -\psi( u_+(z)) & \tto{if} \xi >u_+(z)
\end{array}\right.
\end{equation}
If we define
$$
\widetilde{F}(z,\xi)=\int_0^{\xi}\widetilde{f}(z,t)\textrm{d}t,$$
then
\begin{equation}\label{eq:Ftr3}
\widetilde{F}(z,\xi)=\left\{\begin{array}{lll}
\int\limits_0^{v_-(z)}f(z,t)\ud t \ +&\hspace{-0.3cm} (\xi-v_-(z))f(z, v_-(z)) \ + & \\
\hspace{0.5cm}+\int\limits_{v_-(z)}^{\xi}\psi(t) \ud t& - \  (\xi-v_-(z))\psi( v_-(z))+\frac{1}{p} |\xi|^{p} &\tto{if} \xi<v_-(z),\\
\int\limits_0^{\xi}f(z,t)\ud t+\frac{1}{p} |\xi|^{p} && \tto{if} v_-(z)\leq \xi \leq u_+(z),\\
\int\limits_0^{u_+(z)}f(z,t)\ud t \ +&\hspace{-0.3cm} (\xi-u_+(z))f(z, u_+(z)) \ + & \\
\hspace{0.5cm}+\int\limits_{u_+(z)}^{\xi}\psi(t) \ud t& - \  (\xi-u_+(z))\psi( u_+(z))+\frac{1}{p} |\xi|^{p} & \tto{if} \xi >u_+(z)
\end{array}\right.
\end{equation}
and
\begin{displaymath}
\comz{\widetilde{F}(z,u(z))}=\sum_{i=1}^5\mathrm{I}_i +\frac{1}{p} \|u\|_p^{p},
\end{displaymath}
where
\begin{displaymath}\begin{split}
&I_1=\int\limits_{u<v_-}\left(\int\limits_0^{v_-(z)}f(z,t)\textrm{d}t-v_-(z)f(z,v_-(z))-\frac{\widehat{c}_0}{s} |v_- (z)|^s+\frac{\widehat{c}_1}{r} |v_- (z)|^{r}-v_-(z)\psi(v_-(z)))\right)\ud z,\\
&I_2=\int\limits_{v_-\leq u \leq u_+}\int\limits_0^{u(z)}f(z,t)\textrm{d}t\ud z,\\
&I_3=\int\limits_{u>u_+}\left(\int\limits_0^{u_+(z)}f(z,t)\textrm{d}t-u_+(z)f(z,u_+(z))-\frac{\widehat{c}_0}{s} u_+ (z)^s+\frac{\widehat{c}_1}{r} u_+(z)^{r}-u_+(z)\psi(u_+(z)))\right)\ud z,\\
&I_4=\int\limits_{u<v_-}\left(\frac{\widehat{c}_0}{s} |u(z)|^s-\frac{\widehat{c}_1}{r} |u(z)|^{r}  +u(z)(f(z, v_-(z)) -\psi( v_-(z)))\right)\ud z.\\
&I_5=\int\limits_{u>u_+}\left(\frac{\widehat{c}_0}{s} u(z)^s-\frac{\widehat{c}_1}{r} u(z)^{r}  +u(z)(f(z, u_+(z)) -\psi( u_+(z)))\right)\ud z.\\
\end{split}\end{displaymath}
Consider the $C^1-$functional
\begin{equation}\label{eq:tildephi}
\widetilde{\varphi}(u)=\int_{\Omega}G(\nabla u(z))\ud z+\frac{1}{p}\|u\|^p_p-\int_{\Omega}\widetilde{F}(z,u(z))\ud z, \quad u\in
W^{1,p}(\Omega).
\end{equation}
 Similarly as in the proof of Proposition \ref{stwr:extremality}, we show that $I_1+I_2 +I_3$ is bounded above and thus there exist some constants $\widetilde{c}_{1}>0$, $\widetilde{c}_{2}\in\rzecz$ such that
\begin{equation*}
\widetilde{\varphi}(u)\geq  \frac{c_1}{p(p-1)}\comz{\|\nabla u(z)\|^p} +\widetilde{c}_1\comz{|u(z)|^p}+\widetilde{c}_2, \ u\in\wpo
\end{equation*}
so $\widetilde{\varphi}$ is coercive. Also, let
\begin{equation*}
\widetilde{f}^{\pm}(z,\xi)=\widetilde{f}(z,\pm\xi^{\pm})
\ \tto{and} \
 \widetilde{F}^{\pm}(z,\xi)=\widetilde{F}(z,\pm\xi^{\pm})
\end{equation*}
and consider
\begin{equation*}
\widetilde{\varphi}^{\pm}(u)=\int_{\Omega}G(\nabla u(z))\ud z+\frac{1}{p}\|u\|^p_p-\int_{\Omega}\widetilde{F}^{\pm}(z,u(z))\ud z, \quad u\in
W^{1,p}(\Omega).
\end{equation*}

Reasoning as in the proof of Proposition \ref{stwr:extremality}, we obtain that $\widetilde{\varphi}^{\pm}(u)$ are also coercive.
\begin{lt}\label{lt:cerami}
The functional $\widetilde{\varphi}$ defined by \eqref{eq:tildephi} satisfies the Cerami condition.
\end{lt}
\begin{proof}
Let $\{x_n\}_{n\geq 1}\subseteq \wpo$ be a sequence such that
\begin{equation}\label{eq:ass-cerami}
\widetilde{\varphi}(x_n) \ \to \ c
\quad \tto{and} \quad
(1+\|x_n\|)\widetilde{\varphi}'(x_n) \ \rightarrow \ 0 \ \tto{in} \ \wpo^*.
\end{equation}
The coercivity of $\widetilde{\varphi}$ implies that  $\{x_n\}_{n\geq 1}$ is bounded in $\wpo$, so passing to a subsequence if necessary, we can find $x\in \wpo$ such that
$$
x_n \ \to \ x \ \tto{weakly in} \ \wpo.
$$
Using \eqref{eq:ass-cerami}, we can find a sequence $\{\varepsilon_n\}_{n\geq 1}\subseteq (0, \infty)$ such that $\varepsilon_n\to 0$ and
\begin{equation*}
|\left<A(x_n), h\right>+\comz{|x_n(z)|^{p-2}x_n(z)h(z)}-\comz{\widetilde{f}(z,x_n(z))h(z)}|\leq \varepsilon_n \frac{\|h\|}{ (1+\|x_n\|)}, \ h\in\wpo.
\end{equation*}
Set $h=(x_n-x)\in\wpo$.
By hypothesis $\fullref{ass:hf1:ainf}$ we have
\begin{gather*}
\comz{|x_n(z)|^{p-2}x_n(z)(x_n-x)(z)}-\comz{\widetilde{f}(z,x_n(z))(x_n-x)(z)}\\
= \ \int\limits_{x_n-x<v_-}\!f(z,v_-(z))(x_n-x)(z)\ud z
+\!\!\int\limits_{v_-\leq x_n-x \leq u_+}\!\!f(z,x_n(z))(x_n-x)(z)\ud z\\
+ \!\int\limits_{x_n-x>u_+}\!f(z,u_+(z))(x_n-x)(z)\ud z
 \ \leq \  \comz{\|a_{w_+(z_0)}\|_{\infty}(x_n-x)(z)} \ \to \ 0,
\end{gather*}
as $x_n\to x$ weakly in $\wpo$. Also $\{\frac{\|x_n-x\|}{ (1+\|x_n\|)}\}_{n\geq 1}$ is bounded, by boundedness of $\{x_n\}_{n\geq 1}$ in $\wpo$. Thus
 \begin{equation*}
\lim\limits_{n\to\infty}\left<A(x_n), x_n-x\right> \ = \ 0, 
\end{equation*}
so by Proposition \ref{stwr:propG-a} the sequence $\{x_n\}_{n\geq 1}$ admits a strongly convergent subsequence. 
\end{proof}
\begin{stwr}\label{stwr:ckphi0}
If hypotheses \emph{$\fullref{ass:ha1:zero}$}--\emph{$\fullref{ass:ha1:apG}$} and \emph{$\fullref{ass:hf1:ainf}$}--\emph{$\fullref{ass:hf1:bdd}$} hold and $u=0$ is an isolated critical point of the energy functional for problem \eqref{Problem}, given by
\begin{equation}\label{eq:phi}
\varphi(u)=\comz{G(\nabla u(z))}-\comz{F(z, u(z))}
\end{equation} 
then
$$
C_k(\varphi,0)=0, \quad \forall k\geq 0.
$$
\end{stwr}
\begin{proof}
The idea of the proof follows the proof of Proposition 2.1 in  Jiu-Su \cite{JiuSu2003} (see also Gasi\'nski-Papageorgiou \cite{GasinskiPapageorgiou2014}). From the definition of the critical groups for any $\varrho>0$ such that $K_{\varphi}\cap \varphi^0\cap B_{\varrho}=\{0\}$  we have
$$
C_k(\varphi,0)=H_k(B_{\varrho}\cap \varphi^0, B_{\varrho}\cap \varphi^0\backslash \{0\}), \quad \forall k\geq 0,
$$
 where $B_{\varrho}:=\{u\in\wpo \mid \|u\|\leq \varrho\}$. We know that
$$
C_k(B_{\varrho}, B_{\varrho}\backslash\{0\})=0,   \quad \forall k\geq 0
$$
for any $\varrho>0$ as $B_{\varrho}$ and $B_{\varrho}\backslash\{0\}$ are contractible, because of the fact that $\wpo$ is infinite dimensional (see Propositions 6.24 and 6.25 in Motreanu-Motreanu-Papageorgiou \cite{MotreanuMotreanuPapageorgiou2014}). So our aim is to construct a deformation mapping for $(B_{\varrho}, B_{\varrho}\backslash\{0\})$ and $(B_{\varrho}\cap \varphi^0, B_{\varrho}\cap \varphi^0\backslash\{0\})$. First we prove 
 that for a given $u\in\wpo\backslash \{0\}$ we can find $t_0\in(0,1)$
such that 
\begin{equation}\label{eq:phi_t0}
\varphi(tu)<0 \ \tto{for any} \  t\in(0,t_0).
\end{equation}
Choose $u\in\wpo\backslash \{0\}$. Then, by virtue of hypothesis $\fullref{ass:ha1:mu}$, for a given $\varepsilon>0$ we can find $\delta_{\varepsilon}\in(0,\delta_0]$ such that
$$
G_0(t)\leq \varepsilon t^{\mu}, \ \forall t \in (0, \delta_{\varepsilon}].
$$
There exists some $t_0\in(0,1)$ such that for any $t\in(0,t_0)$ we have $|\nabla (tu)(z)|\leq \min\{\delta_{\varepsilon}, c_+\}$ and $|tu(z)|\leq \min\{\delta_{\varepsilon}, c_+\}$ (see $\fullref{ass:hf1:wc}$) for a.e. $z\in\Omega$. Then
\begin{equation*}
\varphi(tu)= \comz{G(|\nabla (tu)(z)|)}-\comz{F(z, u(z))}\leq t^{\mu}\varepsilon\|\nabla u\|_{\mu}^{\mu}-\hat{c}_2 t^{\mu}\|u\|_{\mu}^{\mu},
\end{equation*}
so choosing $\varepsilon<\frac{\hat{c}_2\|u\|_{\mu}^{\mu}}{\|\nabla u\|_{\mu}^{\mu}}$ we see that \eqref{eq:phi_t0} holds.
\begin{claim}
There exists some $\varrho_1>0$, such that for any $u\in\wpo\backslash \{0\}$ with $\varphi(u)=0$ and $0<\|u\|\leq \varrho_1$ we have
\begin{equation}\label{eq:dtphi1}
\dpd{}{ t}\varphi(tu)\big|_{t=1}>0. 
\end{equation}
\end{claim}
\begin{proofclaim}{}
From \eqref{eq:Fconc} we have that there exist some $\widetilde{c}_3>0$, $\delta_1>0$ such that
\begin{equation*}
F(z, \xi)\geq \widetilde{c}_3 |\xi|^p, \ \tto{for. a.e.} z\in\Omega \tto{and for }  |\xi|\leq \delta_1. 
\end{equation*}
Thus, using \eqref{eq:ass:hf1:eq:bdd}, we obtain that
\begin{equation}\label{eq:Fgeqp-r}
F(z, \xi)\geq \widetilde{c}_3 |\xi|^p-\widetilde{c}_4|\xi|^{r+1} \ \tto{for. a.e.} z\in\Omega \tto{and for all} \xi\in\rzecz 
\end{equation}
with some $\widetilde{c}_4>0$. Now let $u\in\wpo$ be such that $\varphi(u)=0$. Then
\begin{equation}\label{eq:dtphi12}\begin{split}
 \dpd{}{ t}\varphi(tu)\big|_{t=1}&=\langle \varphi'(tu), u \rangle\big|_{t=1}=\comz{\left(a(\nabla u(z)), \nabla u(z)\right)_{\rzecz^N}}-\comz{f(z, u(z))u(z)}\\
& \geq (p-\mu)\comz{G(\nabla u)}
+ \!\!\!\int\limits_{\{|u|\leq\delta_0\}}\!\! \!
\left(\mu F(z, u(z))-f(z, u(z)) \right) \ud z\\
&+ \!\!\!\int\limits_{\{|u|>\delta_0\}}\!\!\!
  \mu \widetilde{c}_3|u(z)|^p \ud  z 
  -\!\!\!\int\limits_{\{|u|>\delta_0\}}\!\!\!
   \left( \mu \widetilde{c}_4|u(z)|^{r+1}+\widehat{c}_3 u(z) +\widehat{c}_3|u(z)|^q\right) \ud  z\\
& \geq \frac{c_3}{p(p-1)}\|\nabla u\|_p^p +  \mu \widetilde{c}_3\|u\|_p^p 
- \!\!\!\int\limits_{\{|u|\leq\delta_0\}}\!\!\!
 \mu \widetilde{c}_3|u(z)|^p \ud z \\
 &- \!\!\!\int\limits_{\{|u|>\delta_0\}}\!\!\!
   \left( \mu \widetilde{c}_4|u(z)|^{r+1}+\widehat{c}_3 |u(z)| +\widehat{c}_3|u(z)|^q\right) \ud  z.
\end{split}\end{equation}
Here we have used \eqref{eq:osz}, \eqref{eq:ass:hf1:eq:flF} and \eqref{eq:Fgeqp-r}. Observe that we can find $\widetilde{M}_0>0$ such that for any $\xi>\delta_0$ we have
$$
 \mu \widetilde{c}_4 \xi^{r+1}+\widehat{c}_3 \xi +\widehat{c}_3\xi^q\leq \widetilde{M}_0 \xi^{r+1}
$$
(recall that $1<q<p\leq r$). So, returning to \eqref{eq:dtphi12}, we obtain
\begin{equation*}\begin{split}
 \dpd{}{ t}\varphi(tu)\big|_{t=1}&\geq \frac{c_1}{p(p-1)}\|\nabla u\|_p^p + \mu \widetilde{c}_3\|u\|_p^p 
- \!\!\!\int\limits_{\{|u|\leq\delta_0\}}\!\!\!
 \mu \widetilde{c}_3|u(z)|^p \ud z 
 - \!\!\!\int\limits_{\{|u|>\delta_0\}}\!\!\!
  \widetilde{M}_0 |u(z)|^{r+1}\ud  z\\
  &\geq \widetilde{c}_5 \|u\|^p - \mu \widetilde{c}_1 \|u\|_p^p-  \widetilde{M}_0\|u\|_{r+1}^{r+1}\\
  &\geq \widetilde{c}_5 \|u\|^p - \widetilde{c}_4\|u\|_{r+1}^{r+1}\geq \widetilde{c}_5 \|u\|^p - \widetilde{c}_6\|u\|^{r+1}
\end{split}\end{equation*}
with some constants $\widetilde{c}_5, \widetilde{c}_6>0$. Here we have used the continuity of the embedding $L^p(\Omega)\subseteq L^{r+1}(\Omega)$. As $p<r+1$, we can find $\varrho_1>0$ such that for $0<\|u\|\leq \varrho_1$ we have \eqref{eq:dtphi1}.
 \end{proofclaim}
Observe that we can take $\varrho_1>0$ small enough to have
\begin{equation*}
K_{\varphi}\cap \varphi^0\cap B_{\varrho_1}=\{0\}.
\end{equation*} 
 The Claim implies that for any $u\in\wpo$ with $\varphi(u)<0$ and $\|u\|\leq \varrho_1$ we have
 \begin{equation}\label{eq:phi_tu}
 \varphi(tu)<0, \ t\in(0,1).
 \end{equation}
 Indeed, if we choose $u\in \wpo$ with $\varphi(u)<0$ and $\|u\|\leq\varrho_1$, then from the continuity of $\varphi$ we have that there exists some $t_1\in[0,1)$ such that
 $$
 \varphi(tu)<0 \ \tto{for} \ t\in(t_1, 1) 
 $$
Now suppose that \eqref{eq:phi_tu} is not true. Then  there exists some  $t_2\in(0,t_1)$ such that
$$
\varphi(t_2 u)=0 \ \tto{with} \ \varphi(tu)<0 \ \tto{for} \ t\in(t_2, 1).
$$
Hence
$$
\dpd{}{t}\varphi(t(t_2u))\big|_{t=1}= \dpd{}{ t}\varphi(tu)\big|_{t=t_2}=\lim\limits_{t\searrow t_2}\frac{\varphi(tu)-\varphi(t_2 u)}{t}\leq 0.
$$
On the other hand, as $\|t_2 u\|\in(0, \varrho_1)$, we have from the Claim that
$$
\dpd{}{ t}\varphi(t(t_2u))\big|_{t=1}>0,  
$$
a contradiction.

Observe that by \eqref{eq:phi_t0}, \eqref{eq:dtphi1} and \eqref{eq:phi_tu}, for any $u\in B_{\varrho_1}$ with $\varphi(u)>0$, there exists a unique $t_*\in(0,1)$ such that $\varphi(t_*u)=0$. Thus the following mapping $T\colon B_{\varrho_1}\to (0,1]$ is well defined
\begin{equation*}
T(u)=\left\{\begin{array}{ll}
1, &\tto{if}  \varphi(u)\leq 0\\
t_* & \tto{if} \varphi(u)>0 \tto{with}  \varphi(t_*u)=0.
\end{array}\right.
\end{equation*}
To show that the mapping $T$ is continuous, take $u\in B_{\varrho_1}$ with $\varphi(u)=0$ and let $\{u_n\}_{n\geq 1}\subseteq B_{\varrho_1} $ be a sequence such that 
$\lim\limits_{n\to\infty}u_n=u $ with $\varphi(u_n)>0$ (because if $\varphi(u_n)\leq 0$ then $T(u_n)=1=T(u)$).
To derive a contradiction, let us suppose that there exists a subsequence, still denoted by $\{u_n\}_{n\geq 1}$, such that
$$
\varphi(u_n) = t_{*,n}<1-\epsilon,
$$
with some $\epsilon\in(0,1)$. Then 
$$
\varphi(tu_n)>0 \ \tto{for all} \ t\in(1-\epsilon, 1],
$$
so, by the continuity of $\varphi$, 
$$\varphi(tu)\geq 0  \ \tto{for all} \ t\in(1-\epsilon, 1].$$
Hence $\varphi(tu)= 0$ (see \eqref{eq:phi_tu}) and thus
$$\dpd{}{ t}\varphi(tu)\big|_{t=1}=0,$$
which contradicts \eqref{eq:dtphi1}. This proves that the mapping $T$ is continuous. It is easy to see that the mapping $\widetilde{h}\colon [0,1]\times  B_{\varrho_1}\to  B_{\varrho_1}$, given by
\begin{equation*}
\widetilde{h}(t,u)=(1-t)u+t T(u)u, \ s\in[0,1], u\in B_{\varrho_1},
\end{equation*}
is a continuous deformation from $(B_{\varrho}, B_{\varrho}\backslash\{0\})$ to $(B_{\varrho}\cap \varphi^0, B_{\varrho}\cap \varphi^0\backslash\{0\})$ with $\widetilde{h}(1, \cdot)\big|_{B_{\varrho}\cap \varphi^0}=id\big|_{B_{\varrho}\cap \varphi^0}$, hence
$$
H_k(B_{\varrho_1}\cap \varphi^0, B_{\varrho_1}\cap \varphi^0\backslash \{0\})=H_k(B_{\varrho_1}, B_{\varrho_1}\backslash \{0\})
$$
(see Corollary 6.15 in Motreanu-Motreanu-Papageorgiou \cite{MotreanuMotreanuPapageorgiou2014}). So $C_k(\varphi,0)=0$, $k\geq 0$.
\end{proof}
\begin{stwr}\label{stwr:nodal}
If hypotheses \emph{$\fullref{ass:ha1:zero}$}--\emph{$\fullref{ass:ha1:apG}$} and \emph{$\fullref{ass:hf1:ainf}$}--\emph{$\fullref{ass:hf1:bdd}$} hold, then problem \eqref{Problem} admits a nodal solution $y_0\in[v_-, u_+]\cap C^1(\Omega)$.
\end{stwr}
\begin{proof}\setcounter{claim_no}{0}
\begin{claim_no}\label{claim:nodal:cpoints}
$K_{\widetilde{\varphi}}\subseteq[v_-,u_+]$, $K_{\widetilde{\varphi}^+}=\{0, u_+\}$, $K_{\widetilde{\varphi}^-}=\{0, v_-\}$.
\end{claim_no}
\begin{proofclaim}{}
Let $u\in K_{\widetilde{\varphi}}$ and suppose $|\{u>u_+\}|_N:=|\{z\in\Omega | u(z)>u_+\}|_N>0$. Then
\begin{equation}\label{eq:arnft}
A(u)+|u|^{p-2} u=\widetilde{f}(u)
\end{equation}
and thus
\begin{equation*}\begin{split}
\left<A(u),(u-u_+)^+\right>
&+\comz{u(z)^{p-1}(u-u_+)^+(z)} =\comz{\widetilde{f}(z, u(z))(u-u_+)^+(z)}\\
&=\comz{f(z, u_+(z))(u-u_+)^+(z)}+\comz{\left(u(z)^{p-1}+\psi(u) -\psi( u_+(z)\right)(u-u_+)^+(z)}\\
&<\left<A(u_+),(u-u_+)^+\right>+\comz{(u(z))^{p-1}(u-u_+)^+(z)}
\end{split}\end{equation*}
(see \eqref{eq:ftr3}), because by Remark \ref{rem:psi}, for $u(z)>u_+(z)$  we have that $\psi(u(z))<\psi(u_+(z))$ (recall that $u_+\geq \xi_0$). This implies
$$
\int_{\{u>u_+\}}\left(a(\nabla u(z))-a(\nabla u_+(z)),\nabla((u-u_+)(z))\right)_{\rzecz^N}= \left<A(u)-A(u_+),(u-u_+)^+\right>< 0,
$$
a contradiction with the strict monotonicity of $a$ (see Proposition \ref{stwr:propG-a}). So for a.e. $z\in\Omega$ we have $u(z)\leq u_+(z)$. Acting with $(v_--u)^+\in \wpo$ on \eqref{eq:arnft} we can show that $u(z)\geq v_-(z)$ for a.e. $z\in\Omega$. In a similar fashion we prove that $K_{\widetilde{\varphi}^+}\subseteq[0, u_+]$, $K_{\widetilde{\varphi}^-}\subseteq[0, v_-]$. As 
$$\widetilde{\varphi}^+\big|_{[0, u_+]}=\varphi\big|_{[0, u_+]}
\tto{and}
\widetilde{\varphi}^-\big|_{[v_-,0]}=\varphi\big|_{[v_-,0]}$$
(see \eqref{eq:ftr3}), from the extremality of $u_+$ and $v_-$ (see Claim \ref{lt:minorantab} in the proof of Proposition \ref{stwr:extremality}) we have that 
 $K_{\widetilde{\varphi}^+}=\{0, u_+\}$ and $K_{\widetilde{\varphi}^-}=\{0, v_-\}$.
\end{proofclaim}
\begin{claim_no}\label{claim:nodal:minimizers}
$u_+$ and $v_-$ are local minimizers of $\widetilde{\varphi}$.
\end{claim_no}
\begin{proofclaim}{}
As $\widetilde{\varphi}_+$ is weakly sequentially lower semicontinuous and coercive, by Theorem \ref{twr:minicoer} we can find $\widetilde{u}_+\in\wpo$ such that
\begin{equation*}
\widetilde{\varphi}_+(\widetilde{u}_+) \ = \ \min\limits_{u\in\wpo}\widetilde{\varphi}_+(u).
\end{equation*}
Similarly as in the proof of Proposition \ref{stwr:extremality} we can show that $\widetilde{u}_+\neq 0$, so  $\widetilde{u}_+=u_+\in\ttt{int}C_+$ (see Claim \ref{claim:nodal:cpoints}). We have that $\widetilde{\varphi}^+\big|_{\ttt{int}C_+}=\widetilde{\varphi}\big|_{\ttt{int}C_+}$, hence $u_+$ is a local $C^1$-minimizer of $\widetilde{\varphi}$. Thus by Theorem \ref{twr:minimizers} $u_+$ is also a $\wpo$-minimizer of $\widetilde{\varphi}$. Analogously, using the functional $\widetilde{\varphi}_-$, we prove that $v_-$ is a~$\wpo$-minimizer of $\widetilde{\varphi}$ as well.
\end{proofclaim}
We can clearly assume that $\widetilde{\varphi}(v_-)\leq \widetilde{\varphi}(u_+)$. If the opposite inequality holds, the argumentation is analogous to this below.
\begin{claim_no}\label{claim:nodal:mpt}
There exists $\varrho_0\in(0,1)$ such that 
\begin{equation*}
\widetilde{\varphi}(v_-)
\leq \widetilde{\varphi}(u_+)
<\inf\{\widetilde{\varphi}(u)\mid \|u-u_+\|=\varrho_0\}=:m_{\varrho_0}
 \ \tto{and} \ 
 \|v_--u_+\|>\varrho_0.
\end{equation*}
\end{claim_no}
\begin{proofclaim}{}
As $u_+$ is an isolated critical point of $\widetilde{\varphi}$ (see Claim \ref{claim:nodal:cpoints}), the part of the proof of Proposition 6 in Motreanu-Motreanu-Papageorgiou \cite{MotreanuMotreanuPapageorgiou2007} is applicable (see also the proof of Theorem 3.4 in Gasi\'nski-Papageorgiou \cite{GasinskiPapageorgiou2009}), with the use of Proposition \ref{stwr:propG-a}.
\end{proofclaim}
Using Claim \ref{claim:nodal:mpt}  and Lemma \ref{lt:cerami} we are able to  apply the mountain pass theorem (Theorem \ref{twr:mpt}) and obtain $y_0\in\wpo$ such that
\begin{equation}\label{eq:y0}
y_0\in K_{\widetilde{\varphi}}
\ \tto{and} \
m_{\varrho_0}\leq \widetilde{\varphi}(y_0).
\end{equation}
We have that $y_0\subseteq [v_-, u_+]$ (see Claim \ref{claim:nodal:cpoints}), so, as $\widetilde{\varphi}\big|_{[v_-,u_+]}=\varphi\big|_{[v_-,u_+]}$, $y_0$ solves problem \eqref{Problem}. Also from Claim \ref{claim:nodal:mpt} and \eqref{eq:y0} we infer that $y_0\neq v_-$ and $y_0\neq u_+$. Thus, if $y_0\neq 0$, the extremality of $v_-$ and $u_+$ implies that $y_0$ is a nodal, nontrivial solution of problem \ref{Problem}. We can assume that $y_0$ is an isolted critical point - otherwise we obtain a whole sequence of nontrivial, nodal solutions, which ends the proof.  Hence, as $y_0$ is of mountain pass type,  it follows that
\begin{equation*}
C_1(\varphi,y_0)\neq 0
\end{equation*}
(see Proposition \ref{stwr:cg1_mpt}). This implies $y_0$ is nontrivial (see Proposition \ref{stwr:ckphi0}).  As before we use the regularity theory to prove that $y_0\in C^1(\Omega)$.
\end{proof}
\begin{uw}
We only need hypothesis $\fullref{ass:hf1:bdd}$ to prove the Claim in Proposition \ref{stwr:ckphi0}. It remains true if instead of $\fullref{ass:hf1:bdd}$ we assume
\begin{assum}{H(f)}{ass:hf1:bdd2}
There exist $\widetilde{M}_1>0$ such that
\begin{equation}\label{eq:ass:hf1:eq:bdd2}
f(z, \xi) \xi\leq \mu F(z, \xi), \ |\xi|> \widetilde{M}_1
\end{equation}
for a.e. $z\in\Omega$.
\end{assum}
Then for  $u\in\wpo$ with  $\varphi(u)=0$ we have
\begin{equation}\label{eq:dtphi13}\begin{split}
 \dpd{}{ t}\varphi(tu)\big|_{t=1}&=\langle \varphi'(tu), u \rangle\big|_{t=1}=\comz{\left(a(\nabla u(z)), \nabla u(z)\right)_{\rzecz^N}}-\comz{f(z, u(z))u(z)}\\
& \geq (p-\mu)\comz{G(\nabla u)}
+ \!\!\!\int\limits_{\{\delta_0<u\leq \widetilde{M}_1\}}\!\!\!
\left(\mu F(z, u(z))-f(z, u(z)) \right) \ud z\\
& \geq \frac{c_1}{p(p-1)}\|\nabla u\|_p^p 
+\!\!\!\int\limits_{\{\delta_0<u\leq \widetilde{M}_1\}}\!\!\!
-\frac{\hat{c}_1}{r}|u(z)|^r \ud z   
- \!\!\!\int\limits_{\{\delta_0<u\leq \widetilde{M}_1\}}\!\!\!
a_{\widetilde{M}_1}(z)|u(z)|\ud z\\
&\geq \widetilde{c}_7\|u\|^p -\widetilde c_8\|u\|_{r+1}^{r+1}\geq \widetilde{c}_7\|u\|^p -\widetilde c_8\|u\|^{r+1}
\end{split}\end{equation}
(see $\fullref{ass:hf1:ainf}$, $\fullref{ass:hf1:eq:fgeqx}$) and then we proceed like in the proof of the Claim.
\end{uw}
Theorem \ref{twr:main_theorem} is a consequence of Proposition \ref{stwr:extremality} and Proposition \ref{stwr:nodal}.


\begin{thebibliography}{99}
\bibitem{aBenciAvenia2000a}
 V. Benci, P. D'Avenia, D. Fortunato, L. Pisani,
 \emph{Solitons in several space dimensions:
Derrick’s problem and
infinitely many solutions},
Arch. Ration. Mech. Anal.,
\textbf{154} (2000), pp. 297--324;
\bibitem{CF1989} \nabbibg{E. Casas, L. A. Fern\'andez}{A Green's formula for quasilinear elliptic operators}{J. Math. Anal. Appl.}{142}{1989}{62--73};
\bibitem{damas98} \nabbibg{L. Damascelli}{Comparison theorems for some quasilinear degenerate elliptic operators and applications to symmetry and monotonicity results}{Ann. Inst. H. Poincar\'e. Analyse
non lin\'eaire}{15}{1998}{493--516};
\bibitem{DiazSaa1987} \nabbibg{J. I. Diaz, J. E. Sa\'a}{Existence et unicit\'e de solutions positives pour certaines \'equations elliptiques quasilin\'eaires}{ C. R. Acad. Sci. Paris S\'er. I}{305}{1987}{521--524};

\bibitem{aDrabek2007a}
P. Dr\'abek,
\emph{The $p$-Laplacian - mascot of nonlinear analysis},
Acta Math. Univ. Comenianae,
\textbf{76} (2000), pp. 85--98;

\bibitem{DunfordSchwartz1988} \nabbibk{N. Dunford, J. T. Schwartz}{Linear Operators. I. General Theory, Volume 7 of Pure and Applied Mathematics}{Wiley}{New York}{1988};
\bibitem{cGasinskiPapageorgiou2008a}
 L. Gasi\'nski,  N. S. Papageorgiou,
 \emph{Existence and multiplicity of solutions for {N}eumann $p$-Laplacian-type equations},
 Adv. Nonlinear Stud.
\textbf{8} (2008), pp. 843--870;
\bibitem{GasinskiPapageorgiou2009}\nabbibg{L. Gasi\'nski, N. S. Papageorgiou}{Nodal and multiple constant sign solutions for resonant p-Laplacian equations with a nonsmooth potential}{Nonlinear Anal.}{71}{2009}{5747--5772};
\bibitem{GP_NA} \nabbibk{L. Gasi\'nski, N. S. Papageorgiou}{Nonlinear Analysis}{Chapman and Hall/CRC Press}{Boca Raton}{2009};
\bibitem{GasinskiPapageorgiou2010} 
 L. Gasi\'nski,  N. S. Papageorgiou,
 \emph{Multiplicity of solutions for nonlinear elliptic equations with combined nonlinearities},
  in: Handbook of Nonconvex Analysis and Applications, Chapter 4, D. Y. Gao and D. Motreanu (eds.), International Press, Boston, 2010, 183--262;
\bibitem{GasinskiPapageorgiou2014} \nabbibg{L. Gasi\'nski, N. S. Papageorgiou}{Dirichlet (p,q)-equations at resonance}{Discrete Contin. Dyn. Syst.}{34}{2014}{2037--2060};
\bibitem{JiuSu2003} \nabbibg{Q. Jiu, J. Su}{Existence and multiplicity results for Dirichlet problems with p-Laplacian}{J. Math. Anal. Appl.}{281}{2003}{587--601};
\bibitem{Klimczak2015a} L. Klimczak,
\emph{Two constant sign solutions for a nonhomogeneous Neumann boundary value problem},
to appear in: Ann. Univ. Paedagog. Crac. Stud. Math., \textbf{14} (2015), pp. 47--62;
\bibitem{Klimczak2015b} L. Klimczak,
\emph{Existence and multiplicity of solutions for a nonhomogeneous Neumann boundary problem},
to appear in: Opuscula Math., \textbf{35} (2015);
\bibitem{Leoni} \nabbibk{G. Leoni}{A First Course in Sobolev Spaces}{American Mathematical Society}{Providence Rhode Island}{2009};
\bibitem{liebe} \nabbibg{G. M. Lieberman}{The natural generalization of the natural conditions of Ladyzhenskaya and Ural'tseva for elliptic equations}{ Commun. Partial Differential
Equations}{16}{1991}{311--361};
\bibitem{monte99} \nabbibg{M. Montenegro}{Strong maximum principles for supersolutions
of quasilinear elliptic equations}{Nonlinear Anal.}{37}{1999}{431--448};
\bibitem{MotreanuMotreanuPapageorgiou2007} \nabbibg{D. Motreanu, V. V. Motreanu, N. Papageorgiou}{A degree theoretic approach for multiple solutions of constant sign for nonlinear elliptic equations}{Manuscripta Math.}{124}{2007}{507--531};
\bibitem{MotreanuMotreanuPapageorgiou2014} \nabbibk{D. Motreanu, V. V. Motreanu, N. Papageorgiou}{Topological and Variational Methods with Applications to Nonlinear Boundary Value Problems}{Springer New York}{New York Heidelberg Dordrecht London}{2014};
\bibitem{MotreanuPapageorgiou2011} \nabbibg{D. Motreanu, N. S. Papageorgiou}{Multiple solutions for nonlinear Neumann problems driven by a nonhomogeneous differential operators}{Proc. Amer. Math.
Soc.}{139}{2011}{3527--3535};
\bibitem{pserr2007} \nabbibk{P. Pucci, J. Serrin}{The Maximum Principle}{Birkh\"auser Verlag AG}{Basel Boston Berlin}{2007};
\bibitem{ziib} \nabbibk{E. Zeidler}{Nonlinear Functional Analysis and Its Applications II/B. Nonlinear Monotone operators}{Springer-Verlag}{New York Berlin Heidelberg London Paris Tokyo}{1990}.

\end{thebibliography}
\end{document}